\documentclass[12pt,leqno]{article}
\usepackage{amsfonts}
\usepackage{amsmath, amsthm, amsfonts, amssymb, color, ulem}
\usepackage{mathrsfs}
\newtheorem{thm}{Theorem}[section]

\newtheorem{lem}[thm]{Lemma}

\theoremstyle{definition}

\newcommand{\scr}[1]{\mathscr #1}
\definecolor{wco}{rgb}{0.5,0.2,0.3}
\renewcommand{\bar}{\overline}
\numberwithin{equation}{section}
\newtheorem{rem}{Remark}[section]

\newcommand{\ua}{\uparrow}

\renewcommand{\hat}{\widehat}
\renewcommand{\tilde}{\widetilde}

\title{{\bf Convergence rates of truncated EM scheme for NSDDEs\thanks{Supported by NSFC(No., 11561027, 11661039), NSF of Jiangxi(No., 20171BAB201010, 20171BCB23046)}}
}

\author{
{\bf  Li Tan$^{a,b}$\, and \, Chenggui Yuan$^c$}\\
 \footnotesize{$^{a}$ School of Statistics, Jiangxi University of Finance and Economics, Nanchang, Jiangxi, 330013, P. R. China}\\
 \footnotesize{$^{b}$ Research Center of Applied Statistics, Jiangxi University of Finance and Economics,}\\
  \footnotesize{ Nanchang, Jiangxi, 330013, P. R. China}\\
\footnotesize{$^c$ Department of Mathematics, Swansea University, Swansea, SA2 8PP, U. K. }\\
\footnotesize{tltanli@126.com, C.Yuan@swansea.ac.uk}}

\begin{document}
\def\R{\mathbb R}  \def\ff{\frac} \def\ss{\sqrt} \def\B{\mathbf
B}
\def\N{\mathbb N} \def\kk{\kappa} \def\m{{\bf m}}
\def\dd{\delta} \def\DD{\Delta} \def\vv{\varepsilon} \def\rr{\rho}
\def\<{\langle} \def\>{\rangle} \def\GG{\Gamma} \def\gg{\gamma}
  \def\nn{\nabla} \def\pp{\partial} \def\EE{\scr E}
\def\d{\text{\rm{d}}} \def\bb{\beta} \def\aa{\alpha} \def\D{\scr D}
  \def\si{\sigma} \def\ess{\text{\rm{ess}}}
\def\beg{\begin} \def\beq{\begin{equation}}  \def\F{\scr F}
\def\Ric{\text{\rm{Ric}}} \def\Hess{\text{\rm{Hess}}}
\def\e{\text{\rm{e}}} \def\ua{\underline a} \def\OO{\Omega}  \def\oo{\omega}
 \def\tt{\tilde} \def\Ric{\text{\rm{Ric}}}
\def\cut{\text{\rm{cut}}} \def\P{\mathbb P} \def\ifn{I_n(f^{\bigotimes n})}
\def\C{\scr C}      \def\aaa{\mathbf{r}}     \def\r{r}
\def\gap{\text{\rm{gap}}} \def\prr{\pi_{{\bf m},\varrho}}  \def\r{\mathbf r}
\def\Z{\mathbb Z} \def\vrr{\varrho} \def\ll{\lambda}
\def\L{\scr L}\def\Tt{\tt} \def\TT{\tt}\def\II{\mathbb I}
\def\i{{\rm in}}\def\Sect{{\rm Sect}}\def\E{\mathbb E} \def\H{\mathbb H}
\def\M{\scr M}\def\Q{\mathbb Q} \def\texto{\text{o}} \def\LL{\Lambda}
\def\Rank{{\rm Rank}} \def\B{\scr B} \def\i{{\rm i}} \def\HR{\hat{\R}^d}
\def\to{\rightarrow}\def\l{\ell}
\def\8{\infty}\def\Y{\mathbb{Y}}

\maketitle

\begin{abstract}
This paper is concerned with strong convergence of the truncated Euler-Maruyama scheme for neutral stochastic differential delay equations driven by Brownian motion and pure jumps respectively. Under local Lipschitz condition, convergence rates of the truncated EM scheme are given.
\end{abstract}


\section{Introduction}
Recently, numerical methods have been widely considered since most equations can not be solved explicitly. Most works on numerical analysis were based on global Lipschitz conditions until Higham et al. \cite{hms02} published a paper studying stochastic differential equations (SDEs) under local Lipschitz conditions. Since then, numerical methods including explicit schemes and implicit schemes have been extensively studied, e.g., strong convergence of  Euler-Maruyama (EM) scheme under local Lipschitz condition \cite{by13}, \cite{ks14}, \cite{m03}, strong convergence of backward EM and $\theta$-EM scheme without global Lipschitz condition \cite{ms13}, \cite{ty162}, \cite{zwh152}, stability of EM scheme under local Lipschitiz condition \cite{hms03}, stability of implicit scheme \cite{mm15}, \cite{zwh15} and the references therein.

Except for the famous EM scheme, some modified Euler-type methods have been developed. For example, in order to deal with SDEs under superlinearly growing and globally one-sided Lipschitz drift coefficient, Hutzenthaler et al. \cite{hjk12} proposed a tamed Euler scheme. Sabanis \cite{s13}, \cite{s15} studied the tamed EM scheme of SDEs under global one-sided Lipschitz and local one-sided Lipschitz coefficient respectively. Based on the classic truncated method, Mao \cite{m15} developed an explicit numerical method called the truncated EM scheme, convergence rate was given in \cite{m16} for SDEs under local Lipschitz condition plus Khasminskii-type condition.

Stochastic differential delay equation (SDDE) is a kind of process that depends on the past states of the system. It plays an important role in theoretical and practical analysis. Moreover, some equations not only depend on past and present states but also involve derivatives with delays as well as the function itself, those equations called neutral stochastic differential delay equations (NSDDEs) also widely exist. Based on \cite{m15}, in this paper, we are going to study strong convergence of the truncated EM scheme for NSDDEs.

This paper is organized as follows: In Section 2, we investigate the convergence rates for NSDDEs driven by Brownian motion, including convergence rate at time $T$ and over a finite time interval under local Lipschitz conditions. In Section 3, NSDDEs driven by pure jumps are considered, convergence rates are obtained under one-sided Lipschitz and superlinearly drift coefficients.

\section{\large Convergence rates for NSDDEs driven by Brownian motion}
\subsection{\normalsize Truncated EM scheme}
Throughout this paper, we let $(\Omega,\mathscr{F},\{\mathscr{F}_t\}_{t\ge 0}, \mathbb{P})$ be a complete filtered probability space satisfying the usual conditions. $(\mathbb{R}^n, \langle\cdot,\cdot\rangle,|\cdot|)$ is an $n$-dimensional Euclidean space. Denote $\mathbb{R}^{n\times d}$ by the set of all $n\times d$ matrices $A$ with trace norm $\|A\|=\sqrt{\mbox{trace}(A^TA)}$, where $A^T$ is the transpose of a matrix $A$. For a given $\tau\in(0,\infty)$, denote $\mathcal{C}([-\tau,0];\mathbb{R}^n)$ by all continuous functions $\zeta$ from $[-\tau,0]$ to $\mathbb{R}^n$ with uniform norm $\|\zeta\|_\infty=\sup_{-\tau\le\theta\le 0}|\zeta(\theta)|$. $W(t)$ is a $d$-dimensional Brownian motion defined on $(\Omega,\mathscr{F},\{\mathscr{F}_t\}_{t\ge 0}, \mathbb{P})$. Consider the following NSDDE on $\mathbb{R}^n$:
\begin{equation}\label{brownian}
\begin{split}
\mbox{d}[X(t)-D(X(t-\tau))]=&b(X(t), X(t-\tau))\mbox{d}t+\sigma(X(t), X(t-\tau))\mbox{d}W(t), t\ge 0
\end{split}
\end{equation}
with initial data $X(t)=\xi(t)\in\mathcal{L}^p_{\mathscr{F}_0}([-\tau,0];\mathbb{R}^n)$ for $t\in[-\tau,0]$, that is, $\xi$ is an $\mathscr{F}_0$-measurable $\mathcal{C}([-\tau,0];\mathbb{R}^n)$-valued random variable with $\mathbb{E}\|\xi\|^p_\infty<\infty$ for $p\ge 2$. Here, $D:\mathbb{R}^n\rightarrow\mathbb{R}^n$, $b:\mathbb{R}^n\times\mathbb{R}^n\rightarrow\mathbb{R}^n$, and $\sigma:\mathbb{R}^n\times\mathbb{R}^n\rightarrow\mathbb{R}^{n\times d}$ are continuous functions. Firstly, we impose some assumptions on coefficients in order to estimate the $p$-th moment of exact and numerical solutions.
\begin{enumerate}
\item[{\bf (A1)}] $D(0)=0$, and there exists a $\kappa\in(0,1)$ such that for $x,y\in\mathbb{R}^n$
\begin{equation*}
|D(x)-D(y)|\le\kappa|x-y|.
\end{equation*}
\item[{\bf (A2)}] For any $R>0$, there exists a positive constant $L_R$ such that
\begin{equation*}
|b(x,y)-b(\bar{x},\bar{y})|\vee\|\sigma(x,y)-\sigma(\bar{x},\bar{y})\|\le L_R(|x-\bar{x}|+|y-\bar{y}|)
\end{equation*}
for $x,y,\bar{x},\bar{y}\in\mathbb{R}^n$ with $|x|\vee|y|\vee|\bar{x}|\vee|\bar{y}|\le R$.
\item[{\bf (A3)}] There exists a pair of constants $p>2$ and $L_1>0$ such that
\begin{equation*}
\langle x-D(y), b(x,y)\rangle+\frac{p-1}{2}\|\sigma(x,y)\|^2\le L_1(1+|x|^2+|y|^2)
\end{equation*}
for $x,y\in\mathbb{R}^n$.
\end{enumerate}

\begin{lem}\label{exactbound}
{\rm Under assumptions (A1)-(A3), there exists a unique solution to \eqref{brownian}. Moreover, for any $T>0$, there exists a positive constant $C$ such that
\begin{equation}\label{exactbound1}
\sup\limits_{0\le t\le T}\mathbb{E}|X(t)|^p\le C,
\end{equation}
and
\begin{equation*}
\mathbb{P}(\tau_R\le T)\le \frac{C}{R^p},
\end{equation*}
where $R$ is a positive real number and $\tau_R$ is a stopping time defined by $\tau_R=\inf\{t\ge 0: |X(t)|\ge R\}$.
}
\end{lem}
\begin{proof}
With (A2), the uniqueness of solution to \eqref{brownian} is guaranteed. With (A1) and (A3), it is easy to show \eqref{exactbound1}. Then by \eqref{exactbound1}, we get
\begin{equation*}
\begin{split}
\mathbb{P}(\tau_R\le T)=\mathbb{E}\left({\bf I}_{\{\tau_R\le T\}}\frac{|X(\tau_R)|^p}{R^p}\right)\le\frac{1}{R^p}\sup\limits_{0\le t \le T}\mathbb{E}|X(t)|^p\le\frac{C}{R^p}.
\end{split}
\end{equation*}
This completes the proof.
\end{proof}

Define a strictly increasing continuous function $f:\mathbb{R}_+\rightarrow\mathbb{R}_+$ such that $f(r)\rightarrow\infty$ as $r\rightarrow\infty$ and for any $r\ge 0$,
\begin{equation*}\label{fr}
\sup\limits_{|x|\vee|y|\le r}(|b(x,y)|\vee\|\sigma(x,y)\|)\le f(r).
\end{equation*}
Obviously, the inverse function $f^{-1}$ is a strictly increasing continuous function from $[f(0),\infty)$ to $\mathbb{R}_+$. Define a strictly decreasing function $g:(0,1]\rightarrow(0,\infty)$ such that
\begin{equation}\label{gdelta}
g(\Delta)\ge1,~~\lim\limits_{\Delta\rightarrow 0}g(\Delta)=\infty,~~\mbox{and} ~~\Delta^{1/4}g(\Delta)\le 1.
\end{equation}
For a given stepsize $\Delta\in(0,1]$, we define the following truncated functions
\begin{equation*}
b_\Delta(x,y)=b\left((|x|\wedge f^{-1}(g(\Delta)))\frac{x}{|x|},(|y|\wedge f^{-1}(g(\Delta)))\frac{y}{|y|}\right),
\end{equation*}
and
\begin{equation*}
\sigma_\Delta(x,y)=\sigma\left((|x|\wedge f^{-1}(g(\Delta)))\frac{x}{|x|},(|y|\wedge f^{-1}(g(\Delta)))\frac{y}{|y|}\right)
\end{equation*}
for $x,y\in\mathbb{R}^n$, where we set $x/|x|=0$ for $x=0$. Obviously, for any $x,y\in\mathbb{R}^n$
\begin{equation}\label{bounddelta}
|b_\Delta(x,y)|\vee\|\sigma_\Delta(x,y)\|\le g(\Delta).
\end{equation}
We now introduce the truncated EM scheme for \eqref{brownian}. Given any time $T>0$, assume that there exist two positive integers such that $\Delta=\frac{\tau}{m}=\frac{T}{M}$. For $k=-m, \cdots, 0$, set $y_{t_k}=\xi(k\Delta)$; For $k=0, 1, \cdots,M-1$, we form
\begin{equation}\label{discrete}
\begin{split}
y_{t_{k+1}}-D(y_{t_{k+1-m}})=&y_{t_k}-D(y_{t_{k-m}})+b_\Delta(y_{t_{k}}, y_{t_{k-m}})\Delta+\sigma_\Delta(y_{t_{k}}, y_{t_{k-m}})\Delta W_{t_k},
\end{split}
\end{equation}
where $t_k=k\Delta$, $\Delta W_{t_k}=W(t_{k+1})-W(t_k)$. Define the corresponding continuous form of \eqref{discrete} as follows: set $Y(t)=\xi(t)$ for $t\in[-\tau,0]$, and for any $t\in[0,T]$, we form
\begin{equation}\label{continuous}
\begin{split}
Y(t)-D(\bar{Y}(t-\tau))=&\xi(0)-D(\xi(-\tau))+\int_0^tb_\Delta(\bar{Y}(s), \bar{Y}(s-\tau))\mbox{d}s\\
&+\int_0^t\sigma_\Delta(\bar{Y}(s), \bar{Y}(s-\tau))\mbox{d}W(s),
\end{split}
\end{equation}
where $\bar{Y}(t)$ is defined by
\begin{equation*}
\bar{Y}(t):=y_{t_k} \quad \mbox{for} \quad t\in[t_k, t_{k+1}),
\end{equation*}
thus $\bar{Y}(t-\tau)=y_{t_{k-m}}$.

To obtain the $p$-th moment of numerical solutions, we impose another assumption.
\begin{enumerate}
\item[{\bf (A4)}] There exists a pair of constants $p>2$ and $\bar{L}_1>0$ such that
\begin{equation*}
\langle x-D(y), b_\Delta(x,y)\rangle+\frac{p-1}{2}\|\sigma_\Delta(x,y)\|^2\le \bar{L}_1(1+|x|^2+|y|^2)
\end{equation*}
for any $x,y\in\mathbb{R}^n$.
\end{enumerate}

\begin{rem}\label{example1}
{\rm Consider the special case with $D(\cdot)=0$, if we choose a number $\Delta^*\in(0,1]$ and a strictly decreasing function $g:(0,\Delta^*]\rightarrow(0,\infty)$ such that for any $\Delta\in(0,\Delta^*]$
\begin{equation*}
f(2)\le g(\Delta^*),~~\lim\limits_{\Delta\rightarrow 0}g(\Delta)=\infty,~~\mbox{and} ~~\Delta^{1/4}g(\Delta)\le 1.
\end{equation*}
Then, we can show that for any functions satisfying (A3), we all have
\begin{equation}\label{remmm}
\langle x, b_\Delta(x,y)\rangle+\frac{p-1}{2}\|\sigma_\Delta(x,y)\|^2\le \bar{L}_1(1+|x|^2+|y|^2)
\end{equation}
for any $x,y\in\mathbb{R}^n$, where $\bar{L}_1=2L_1$. That is, if SDDEs with no neutral term are considered, assumption (A4) can be eliminated. The proof of \eqref{remmm} is shown in Appendix.
}
\end{rem}

\begin{rem}\label{example2}
{\rm Generally, there are lots of functions such that $D(x),b(x,y),\sigma(x,y)$ satisfy (A1)-(A3) and the corresponding $b_\Delta(x,y),\sigma_\Delta(x,y)$ satisfy (A4). For example, if we define
$D(y)=-ay$, $b(x,y)=x+ay-(x+ay)^3$, $\sigma(x,y)=|x+ay|^{\frac{3}{2}}$ for $a\in(-1,1)$, $x,y\in\mathbb{R}$, we can show that $D$ satisfies (A1), $b$ and $\sigma$ are locally Lipschitz. In fact, for $|x|\vee|y|\vee|\bar{x}|\vee|\bar{y}|\le R$,
\begin{equation*}
\begin{split}
|b(x,y)-b(\bar{x},\bar{y})|=|x+ay-(x+ay)^3-\bar{x}-a\bar{y}+(\bar{x}+a\bar{y})^3|\le(1+6R^2)(|x-\bar{x}|+|y-\bar{y}|),
\end{split}
\end{equation*}
and by mean value theorem in $\mathbb{R}^2$,
\begin{equation*}
\begin{split}
\|\sigma(x,y)-\sigma(\bar{x},\bar{y})\|=\left||x+ay|^{\frac{3}{2}}-|\bar{x}+a\bar{y}|^{\frac{3}{2}}|\right|\le3\sqrt{2R}(|x-\bar{x}|+|y-\bar{y}|).
\end{split}
\end{equation*}
Moreover, for any $p>3$, we have
\begin{equation*}
\langle x-D(y), b(x,y)\rangle+\frac{p-1}{2}\|\sigma(x,y)\|^2=|x+ay|^2\left(1+\frac{(p-1)^2}{16}\right)\le C(1+|x|^2+|y|^2).
\end{equation*}
If we choose a number $\Delta^*\in(0,1]$ and a strictly decreasing function $g:(0,\Delta^*]\rightarrow(0,\infty)$ such that for any $\Delta\in(0,\Delta^*]$, $f(2)\le g(\Delta^*)$. Then, similar to that of Remark \ref{example1}, we immediately obtain (A4). Another example is $D(y)=\frac{1}{2}\sin y$, $b(x,y)=x-x^3+\cos y$, $\sigma(x,y)=|x|^{\frac{3}{2}}$, one can also show that (A1)-(A4) is satisfied. The detailed proof is shown in Appendix. In fact, we can further show those $D(x),b(x,y),\sigma(x,y)$ satisfy assumptions (A5)-(A7) below.
}
\end{rem}

\begin{lem}\label{exa-num}
{\rm For any $\hat{p}>0$, there exists a positive constant $C$ independent of $\Delta$ such that
\begin{equation}\label{mediu}
\sup\limits_{t\in[0,T]}\mathbb{E}|Y(t)-\bar{Y}(t)|^{\hat{p}}\le C\Delta^{\hat{p}/2}(g(\Delta))^{\hat{p}},
\end{equation}
and
\begin{equation*}
\lim\limits_{\Delta\rightarrow 0}\sup\limits_{t\in[0,T]}\mathbb{E}|Y(t)-\bar{Y}(t)|^{\hat{p}}=0.
\end{equation*}
}
\end{lem}
\begin{proof}
For any $t\in[0,T]$, there exists a $k$ such that $t\in[t_k,t_{k+1})$, thus by \eqref{bounddelta}, together with the H\"{o}lder inequality and the Burkholder-Davis-Gundy (BDG) inequality, for ${\hat{p}}\ge 2$, we derive from \eqref{continuous} that
\begin{equation*}
\begin{split}
\mathbb{E}|Y(t)-\bar{Y}(t)|^{\hat{p}}\le& C\mathbb{E}\left|\int_{t_k}^tb_\Delta(\bar{Y}(s),\bar{Y}(s-\tau)){\mbox d}s\right|^{\hat{p}}+\mathbb{E}\left|\int_{t_k}^t\sigma_\Delta(\bar{Y}(s),\bar{Y}(s-\tau)){\mbox d}W(s)\right|^{\hat{p}}\\
\le&C\Delta^{{\hat{p}}-1}\mathbb{E}\int_{t_k}^t|b_\Delta(\bar{Y}(s),\bar{Y}(s-\tau))|^{\hat{p}}{\mbox d}s+C\Delta^{\frac{{\hat{p}}-2}{2}}\mathbb{E}\int_{t_k}^t|\sigma_\Delta(\bar{Y}(s),\bar{Y}(s-\tau))|^{\hat{p}}{\mbox d}s\\
\le&C\Delta^{{\hat{p}}/2}(g(\Delta))^{\hat{p}}.
\end{split}
\end{equation*}
For ${\hat{p}}\in(0,2)$, by the H\"{o}lder inequality again, \eqref{mediu} follows. Since $\Delta^{{\hat{p}}/2}(g(\Delta))^{\hat{p}}\le\Delta^{{\hat{p}}/4}$, the second part is shown by taking limits on both sides of \eqref{mediu}.

\end{proof}

\begin{lem}\label{numbound}
{\rm Let (A1) and (A4) hold, then there exists a positive constant $C$ independent of $\Delta$ such that
\begin{equation}\label{ytp}
\sup\limits_{t\in[0,T]}\mathbb{E}|Y(t)|^p\le C,
\end{equation}
and
\begin{equation*}
\mathbb{P}(\rho_R\le T)\le \frac{C}{R^p},
\end{equation*}
where $R$ is a positive real number and $\rho_R$ is a stopping time defined by $\rho_R=\inf\{t\ge 0: |Y(t)|\ge R\}$.
}
\end{lem}
\begin{proof}
Application of the It\^{o} formula yields
\begin{equation}\label{ztp}
\begin{split}
&\mathbb{E}|Y(t)-D(\bar{Y}(t-\tau))|^p\le|\xi(0)-D(\xi(-\tau))|^p+p\mathbb{E}\int_0^t|Y(s)-D(\bar{Y}(s-\tau))|^{p-2}\\
&~~~\cdot\bigg[\langle Y(s)-D(\bar{Y}(s-\tau)),b_\Delta(\bar{Y}(s), \bar{Y}(s-\tau))\rangle+\frac{p-1}{2}\|\sigma_\Delta(\bar{Y}(s), \bar{Y}(s-\tau))\|^2\bigg]{\mbox d}s\\
\le&|\xi(0)-D(\xi(-\tau))|^p+p\mathbb{E}\int_0^t|Y(s)-D(\bar{Y}(s-\tau))|^{p-2}\langle Y(s)-\bar{Y}(s),b_\Delta(\bar{Y}(s), \bar{Y}(s-\tau))\rangle\\
&+p\mathbb{E}\int_0^t|Y(s)-D(\bar{Y}(s-\tau))|^{p-2}\bigg[\langle \bar{Y}(s)-D(\bar{Y}(s-\tau)),b_\Delta(\bar{Y}(s), \bar{Y}(s-\tau))\rangle\\
&+\frac{p-1}{2}\|\sigma_\Delta(\bar{Y}(s), \bar{Y}(s-\tau))\|^2\bigg]{\mbox d}s\\
=:&|\xi(0)-D(\xi(-\tau))|^p+J_1(t)+J_2(t).
\end{split}
\end{equation}
With (A1) and Lemma \ref{exa-num}, we can easily derive from \eqref{bounddelta} that
\begin{equation*}
\begin{split}
J_1(t)\le&C\mathbb{E}\int_0^t|Y(s)-D(\bar{Y}(s-\tau))|^{p}{\mbox d}s+C\mathbb{E}\int_0^t|Y(s)-\bar{Y}(s)|^{p/2}|b_\Delta(\bar{Y}(s), \bar{Y}(s-\tau))|^{p/2}{\mbox d}s\\
\le&C\int_0^t\sup\limits_{0\le u\le s}\mathbb{E}|Y(u)|^{p}{\mbox d}s+C\Delta^{p/4}(g(\Delta))^p+C.
\end{split}
\end{equation*}
By (A1) and (A4), we have
\begin{equation*}
\begin{split}
J_2(t)\le&p\bar{L}_1\mathbb{E}\int_0^t|Y(s)-D(\bar{Y}(s-\tau))|^{p-2}(1+|\bar{Y}(s)|^2+|\bar{Y}(s-\tau)|^2)\\
\le&C\mathbb{E}\int_0^t(|Y(s)|^{p-2}+|\bar{Y}(s-\tau)|^{p-2})(1+|\bar{Y}(s)|^2+|\bar{Y}(s-\tau)|^2){\mbox d}s\\
\le&C+C\int_0^t\sup\limits_{0\le u\le s}\mathbb{E}|Y(u)|^{p}{\mbox d}s.
\end{split}
\end{equation*}
Substituting $J_1(t)$ and $J_2(t)$ into \eqref{ztp}, we have
\begin{equation*}
\begin{split}
\sup\limits_{0\le u\le t}\mathbb{E}|Y(u)-D(\bar{Y}(u-\tau))|^p\le C+C\int_0^t\sup\limits_{0\le u\le s}\mathbb{E}|Y(u)|^{p}{\mbox d}s,
\end{split}
\end{equation*}
where we have used \eqref{gdelta}. Consequently,
\begin{equation*}
\begin{split}
\sup\limits_{0\le u\le t}\mathbb{E}|Y(u)|^{p}\le C+C\sup\limits_{0\le u\le t}\mathbb{E}|Y(u)-D(\bar{Y}(u-\tau))|^p\le C+C\int_0^t\sup\limits_{0\le u\le s}\mathbb{E}|Y(u)|^{p}{\mbox d}s.
\end{split}
\end{equation*}
The Gronwall inequality then leads to \eqref{ytp}. Moreover, by \eqref{ytp}, we get
\begin{equation*}
\mathbb{P}(\rho_R\le T)=\mathbb{E}\left({\bf I}_{\{\rho_R\le T\}}\frac{|Y(\rho_R)|^p}{R^p}\right)\le\frac{1}{R^p}\sup\limits_{t\in[0,T]}\mathbb{E}|Y(t)|^p\le\frac{C}{R^p}.
\end{equation*}
This completes the proof.
\end{proof}

\subsection{\normalsize Strong convergence rate at time $T$}
In order to obtain the strong convergence rate, we need more assumptions.
\begin{enumerate}
\item[{\bf (A5)}] There exist positive constants $q\ge 2$ and $L_2$ such that
\begin{equation*}
\langle x-D(y)-\bar{x}+D(\bar{y}), b(x,y)-b(\bar{x},\bar{y})\rangle+\frac{q-1}{2}\|\sigma(x,y)-\sigma(\bar{x},\bar{y})\|^2\le L_2(|x-\bar{x}|^2+|y-\bar{y}|^2)
\end{equation*}
for $x,y,\bar{x},\bar{y}\in\mathbb{R}^n$.
\item[{\bf (A6)}] There exist positive constants $l\ge1$ and $L_3$ such that
\begin{equation*}
|b(x,y)|\le L_3(1+|x|^l+|y|^l)
\end{equation*}
for $x,y\in\mathbb{R}^n$.
\end{enumerate}

\begin{lem}\label{x-ynur}
{\rm Let (A1)-(A6) hold, $R$ is a real number and let $\Delta$ be sufficiently small such that $\|\xi\|_{\infty}<R\le f^{-1}(g(\Delta))$, then there exists a positive constant $C$ independent of $\Delta$ such that for $q<p$ and $ql<2p$,
\begin{equation*}
\sup\limits_{0\le t\le T}\mathbb{E}|X(t\wedge\nu_R)-Y(t\wedge\nu_R)|^q\le C\Delta^{q/4}(g(\Delta))^{q/2},
\end{equation*}
where $\nu_R:=\tau_R\wedge\rho_R$, and $\tau_R$, $\rho_R$ is defined as in Lemma \ref{exactbound} and \ref{numbound}.
}
\end{lem}
\begin{proof}
Denote $e(t)=X(t)-D(X(t-\tau))-Y(t)+D(\bar{Y}(t-\tau))$. By the It\^{o} formula, for any $t\in[0,T]$,
\begin{equation*}
\begin{split}
&\mathbb{E}|e(t\wedge\nu_R)|^q\le q\mathbb{E}\int_0^{t\wedge\nu_R}|e(s)|^{q-2}\bigg[\langle e(s),b(X(s),X(s-\tau))-b_\Delta(\bar{Y}(s),\bar{Y}(s-\tau))\rangle\\
&+\frac{q-1}{2}\|\sigma(X(s),X(s-\tau))-\sigma_\Delta(\bar{Y}(s),\bar{Y}(s-\tau))\|^2\bigg]{\mbox d}s\\
\le&q\mathbb{E}\int_0^{t\wedge\nu_R}|e(s)|^{q-2}\bigg[\langle X(s)-D(X(s-\tau))-\bar{Y}(s)+D(\bar{Y}(s-\tau)),b(X(s),X(s-\tau))\\
&-b_\Delta(\bar{Y}(s),\bar{Y}(s-\tau))\rangle+\frac{q-1}{2}\|\sigma(X(s),X(s-\tau))-\sigma_\Delta(\bar{Y}(s),\bar{Y}(s-\tau))\|^2\bigg]{\mbox d}s\\
&+q\mathbb{E}\int_0^{t\wedge\nu_R}|e(s)|^{q-2}\langle \bar{Y}(s)-Y(s),b(X(s),X(s-\tau))-b_\Delta(\bar{Y}(s),\bar{Y}(s-\tau))\rangle{\mbox d}s\\
=:&H_1(t)+H_2(t).
\end{split}
\end{equation*}
Since for $0\le s\le t\wedge\nu_R$, one has $|\bar{Y}(s)|\vee|\bar{Y}(s-\tau)|\le R\le f^{-1}(g(\Delta))$. With the definition of $b_\Delta$ and $\sigma_\Delta$, we then have
\begin{equation*}
\begin{split}
b_\Delta(\bar{Y}(s),\bar{Y}(s-\tau))=b(\bar{Y}(s),\bar{Y}(s-\tau)),
\end{split}
\end{equation*}
and
\begin{equation*}
\begin{split}
\sigma_\Delta(\bar{Y}(s),\bar{Y}(s-\tau))=\sigma(\bar{Y}(s),\bar{Y}(s-\tau))
\end{split}
\end{equation*}
for $0\le s\le t\wedge\nu_R$. Thus, we derive from (A5) and Lemma \ref{exa-num} that
\begin{equation*}
\begin{split}
H_1(t)\le& qL_2\mathbb{E}\int_0^{t\wedge\nu_R}|e(s)|^{q-2}(|X(s)-\bar{Y}(s)|^2+|X(s-\tau)-\bar{Y}(s-\tau)|^2){\mbox d}s\\
\le&C\mathbb{E}\int_0^{t\wedge\nu_R}|e(s)|^{q}{\mbox d}s+C\mathbb{E}\int_0^{t\wedge\nu_R}(|X(s)-\bar{Y}(s)|^q+|X(s-\tau)-\bar{Y}(s-\tau)|^q){\mbox d}s \\
\le&C\int_0^{t}\mathbb{E}|X(s\wedge\nu_R)-Y(s\wedge\nu_R)|^q{\mbox d}s+C\Delta^{q/2}(g(\Delta))^q.
\end{split}
\end{equation*}
Moreover, by (A1), (A6) and Lemmas \ref{exactbound}-\ref{numbound}, we derive
\begin{equation*}
\begin{split}
H_2(t)\le&C\mathbb{E}\int_0^{t\wedge\nu_R}|e(s)|^{q}{\mbox d}s+C\mathbb{E}\int_0^{t\wedge\nu_R}|\bar{Y}(s)-Y(s)|^{\frac{q}{2}}\\
&\cdot|b(X(s),X(s-\tau))-b(\bar{Y}(s),\bar{Y}(s-\tau))|^{\frac{q}{2}}{\mbox d}s\\
\le&C\mathbb{E}\int_0^{t\wedge\nu_R}|e(s)|^q{\mbox d}s+C\mathbb{E}\int_0^{t\wedge\nu_R}|\bar{Y}(s)-Y(s)|^{\frac{q}{2}}\\
&\cdot(1+|X(s)|^l+|X(s-\tau)|^l+|\bar{Y}(s)|^l+|\bar{Y}(s-\tau)|^l)^{\frac{q}{2}}{\mbox d}s\\
\le&C\mathbb{E}\int_0^{t}|e(s)|^{q}{\mbox d}s+C\int_0^{t\wedge\nu_R}[\mathbb{E}|Y(s\wedge\nu_R)-\bar{Y}(s\wedge\nu_R)|^{\frac{pq}{2p-ql}}]^{\frac{2p-ql}{2p}}\\
&[\mathbb{E}(1+|X(s\wedge\nu_R)|^p+|X(s\wedge\nu_R-\tau)|^p+|\bar{Y}(s\wedge\nu_R)|^p+|\bar{Y}(s\wedge\nu_R-\tau)|^p)]^{\frac{ql}{2p}}{\mbox d}s\\
\le&C\mathbb{E}\int_0^{t}|X(s\wedge\nu_R)-Y(s\wedge\nu_R)|^q{\mbox d}s+C\Delta^{q/4}(g(\Delta))^{q/2}.
\end{split}
\end{equation*}
Thus, the estimation of $H_1(t)$-$H_2(t)$ leads to
\begin{equation*}
\begin{split}
\mathbb{E}|e(t\wedge\nu_R)|^q\le C\int_0^{t}\mathbb{E}|X(s\wedge\nu_R)-Y(s\wedge\nu_R)|^q{\mbox d}s+C\Delta^{q/4}(g(\Delta))^{q/2}.
\end{split}
\end{equation*}
Consequently,
\begin{equation*}
\begin{split}
&\mathbb{E}|X(t\wedge\nu_R)-Y(t\wedge\nu_R)|^q\le C\mathbb{E}|e(t\wedge\nu_R)|^q\\
\le &C\Delta^{q/4}(g(\Delta))^{q/2}+C\int_0^t\mathbb{E}|X(s\wedge\nu_R)-Y(s\wedge\nu_R)|^q{\mbox d}s.
\end{split}
\end{equation*}
Application of the Gronwall inequality yields the desired result.
\end{proof}

\begin{thm}\label{theorem1}
{\rm Let (A1)-(A6) hold, $\Delta$ is sufficiently small such that
\begin{equation*}
f\left([\Delta^{q/4}(g(\Delta))^{q/2}]^{\frac{-1}{p-q}}\right)\le g(\Delta),
\end{equation*}
then there exists a positive constant $C$ independent of $\Delta$ such that for $q<p$ and $ql<2p$,
\begin{equation*}
\mathbb{E}|X(T)-Y(T)|^q\le C\Delta^{q/4}(g(\Delta))^{q/2}
\end{equation*}
for all $T>0$.
}
\end{thm}
\begin{proof}
By the Young inequality, we derive that for any $\eta>0$,
\begin{equation}\label{eTq}
\begin{split}
&\mathbb{E}|X(T)-Y(T)|^q=\mathbb{E}(|X(T)-Y(T)|^q{\bf I}_{\{\tau_R>T,\rho_R>T\}})+\mathbb{E}(|X(T)-Y(T)|^q{\bf I}_{\{\tau_R\le T~{\rm or}~\rho_R\le T\}})\\
\le&\mathbb{E}(|X(T)-Y(T)|^q{\bf I}_{\{\nu_R>T\}})+\frac{q\eta}{p}\mathbb{E}|X(T)-Y(T)|^p+\frac{p-q}{p\eta^{\frac{q}{p-q}}}\mathbb{P}(\tau_R\le T~{\rm or}~\rho_R\le T).
\end{split}
\end{equation}
By Lemmas \ref{exactbound} and \ref{numbound}, we derive
\begin{equation}\label{eTqp}
\mathbb{E}|X(T)-Y(T)|^p\le C\mathbb{E}|X(T)|^p+C\mathbb{E}|Y(T)|^p\le C.
\end{equation}
On the other hand, Lemmas \ref{exactbound} and \ref{numbound} show that
\begin{equation}\label{ptaur}
\begin{split}
\mathbb{P}(\tau_R\le T~{\rm or}~\rho_R\le T)\le\mathbb{P}(\tau_R\le T)+\mathbb{P}(\rho_R\le T)\le\frac{2C}{R^p}.
\end{split}
\end{equation}
Consequently, it follows from \eqref{eTq}-\eqref{ptaur} that
\begin{equation*}
\begin{split}
\mathbb{E}|X(T)-Y(T)|^q\le&\mathbb{E}|X(T\wedge\nu_R)-Y(T\wedge\nu_R)|^q+\frac{q\eta C}{p}+\frac{2(p-q)C}{p\eta^{\frac{q}{p-q}}R^p}.
\end{split}
\end{equation*}
Choosing $\eta=\Delta^{q/4}(g(\Delta))^{q/2}$ and $R=[\Delta^{q/4}(g(\Delta))^{q/2}]^{\frac{-1}{p-q}}$, noticing that
\begin{equation*}
R=[\Delta^{q/4}(g(\Delta))^{q/2}]^{\frac{-1}{p-q}}\le f^{-1}(g(\Delta)),
\end{equation*}
applying Lemma \ref{x-ynur}, the desired result will be obtained.
\end{proof}

\begin{rem}\label{song}
{\rm For the second example in Remark \ref{example2}, we see (A1)-(A6) are satisfied with $p=3,q=2,l=3$, if we let $g(\Delta)=\Delta^{-\epsilon}$ for $\epsilon\in(0,\frac{1}{4}]$, we see the order of $L^2$-convergence is close to 1/2. That is,
\begin{equation*}
\begin{split}
\mathbb{E}|X(T)-Y(T)|^2=\mathcal{O}(\Delta^{1/2-\epsilon}).
\end{split}
\end{equation*}

}
\end{rem}

Remark \ref{song} shows that under assumptions (A1)-(A6), the order of $L^2$-convergence for the example is close to 1/2. In fact, if we replace (A2), (A6) with the following (A7), the strong convergence rate can be improved.
\begin{enumerate}
\item[{\bf (A7)}] There exist positive constants $l$ and $L_4$ such that
\begin{equation*}
|b(x,y)-b(\bar{x},\bar{y})|+\|\sigma(x,y)-\sigma(\bar{x},\bar{y})\|\le L_4(1+|x|^l+|\bar{x}|^l+|y|^l+|\bar{y}|^l)(|x-\bar{x}|+|y-\bar{y}|)
\end{equation*}
for $x,y,\bar{x},\bar{y}\in\mathbb{R}^n$.
\end{enumerate}

\begin{rem}
{\rm Assumption (A7) implies (A2) and (A6). Thus, Lemma \ref{exactbound} still holds with (A2) replaced by (A7).

}
\end{rem}

\begin{lem}
{\rm Let (A1), (A3)-(A5) and (A7) hold, $R$ is a real number and let $\Delta$ be sufficiently small such that $\|\xi\|_{\infty}<R\le f^{-1}(g(\Delta))$, then there exists a positive constant $C$ independent of $\Delta$ such that for $q<p$, $ql<2p$, and $r\in[2,q)$
\begin{equation*}
\sup\limits_{0\le t\le T}\mathbb{E}|X(t\wedge\nu_R)-Y(t\wedge\nu_R)|^{r}\le C\Delta^{r/2}(g(\Delta))^{r},
\end{equation*}
where $\nu_R:=\tau_R\wedge\rho_R$, and $\tau_R$, $\rho_R$ is the same as before.
}
\end{lem}
\begin{proof}
By the It\^{o} formula, we have
\begin{equation}\label{barh}
\begin{split}
&\mathbb{E}|e(t\wedge\nu_R)|^{r}\le r\mathbb{E}\int_0^{t\wedge\nu_R}|e(s)|^{r-2}\bigg[\langle e(s),b(X(s),X(s-\tau))-b_\Delta(\bar{Y}(s),\bar{Y}(s-\tau))\rangle\\
&+\frac{r-1}{2}\|\sigma(X(s),X(s-\tau))-\sigma_\Delta(\bar{Y}(s),\bar{Y}(s-\tau))\|^2\bigg]{\mbox d}s\\
\le&C\mathbb{E}\int_0^{t\wedge\nu_R}|e(s)|^{r-2}\bigg[\langle e(s),b(X(s),X(s-\tau))-b(Y(s),\bar{Y}(s-\tau))\rangle\\
&~~~~+\frac{q-1}{2}\|\sigma(X(s),X(s-\tau))-\sigma(Y(s),\bar{Y}(s-\tau))\|^2\bigg]{\mbox d}s\\
&+C\mathbb{E}\int_0^{t\wedge\nu_R}|e(s)|^{r-2}\bigg[\langle e(s),b(Y(s),\bar{Y}(s-\tau))-b(\bar{Y}(s),\bar{Y}(s-\tau))\rangle\\
&~~~~+\|\sigma(Y(s),\bar{Y}(s-\tau))-\sigma(\bar{Y}(s),\bar{Y}(s-\tau))\|^2\bigg]{\mbox d}s\\
=:&\bar{H}_1(t)+\bar{H}_2(t).
\end{split}
\end{equation}
By (A5) and Lemma \ref{exa-num} that
\begin{equation*}
\begin{split}
\bar{H}_1(t)\le& C\mathbb{E}\int_0^{t\wedge\nu_R}|e(s)|^{r-2}(|X(s)-Y(s)|^2+|X(s-\tau)-\bar{Y}(s-\tau)|^2){\mbox d}s\\
\le&C\mathbb{E}\int_0^{t\wedge\nu_R}|e(s)|^{r}{\mbox d}s+C\mathbb{E}\int_0^{t\wedge\nu_R}(|X(s)-Y(s)|^{r}+|X(s-\tau)-\bar{Y}(s-\tau)|^{r}){\mbox d}s \\
\le&C\int_0^{t}\mathbb{E}|X(s\wedge\nu_R)-Y(s\wedge\nu_R)|^{r}{\mbox d}s+C\Delta^{r/2}(g(\Delta))^{r}.
\end{split}
\end{equation*}
Moreover, by (A7), Lemmas \ref{exactbound}-\ref{numbound}, we derive that
\begin{equation*}
\begin{split}
\bar{H}_2(t)\le&C\mathbb{E}\int_0^{t\wedge\nu_R}|e(s)|^{r}{\mbox d}s+C\mathbb{E}\int_0^{t\wedge\nu_R}|b(Y(s),\bar{Y}(s-\tau))-b(\bar{Y}(s),\bar{Y}(s-\tau))|^{r}{\mbox d}s\\
&+C\mathbb{E}\int_0^{t\wedge\nu_R}\|\sigma(Y(s),\bar{Y}(s-\tau))-\sigma(\bar{Y}(s),\bar{Y}(s-\tau))\|^{r}{\mbox d}s\\
\le&C\mathbb{E}\int_0^{t}|X(s\wedge\nu_R)-Y(s\wedge\nu_R)|^{r}{\mbox d}s+C\Delta^{r/2}(g(\Delta))^{r}.
\end{split}
\end{equation*}
Substitute $\bar{H}_1(t)$-$\bar{H}_2(t)$ into \eqref{barh}, we get
\begin{equation*}
\begin{split}
\mathbb{E}|e(t\wedge\nu_R)|^{r}\le C\int_0^{t}\mathbb{E}|X(s\wedge\nu_R)-Y(s\wedge\nu_R)|^{r}{\mbox d}s+C\Delta^{r/2}(g(\Delta))^{r}.
\end{split}
\end{equation*}
Consequently,
\begin{equation*}
\begin{split}
&\mathbb{E}|X(t\wedge\nu_R)-Y(t\wedge\nu_R)|^{r}\le C\mathbb{E}|e(t\wedge\nu_R)|^{r}\\
\le &C\Delta^{r/2}(g(\Delta))^{r}+C\int_0^t\mathbb{E}|X(s\wedge\nu_R)-Y(s\wedge\nu_R)|^{r}{\mbox d}s.
\end{split}
\end{equation*}
Finally, application of the Gronwall inequality yields the desired result.
\end{proof}

\begin{thm}
{\rm Let (A1), (A3)-(A5) and (A7) hold with $q<p$, $ql<2p$ and $r\in[2,q)$, assume
\begin{equation*}
f\left([\Delta^{r/2}(g(\Delta))^{r}]^{\frac{-1}{p-r}}\right)\le g(\Delta),
\end{equation*}
then there exists a positive constant $C$ independent of $\Delta$ such that
\begin{equation*}
\mathbb{E}|X(T)-Y(T)|^{r}\le C\Delta^{r/2}(g(\Delta))^{r}.
\end{equation*}
}
\end{thm}
\begin{proof}
We omit the proof since it is similar to that of Theorem \ref{theorem1}.
\end{proof}

\begin{rem}
{\rm Considering the second example in Remark \ref{example2},  (A1), (A3)-(A5) and (A7) hold with $p=3,q=2,l=3$. Choose $g(\Delta)=\Delta^{-\epsilon/2}$ for $\epsilon\in(0,\frac{1}{2}]$, then the order of $L^2$-convergence is close to 1. That is,
\begin{equation*}
\begin{split}
\mathbb{E}|X(T)-Y(T)|^2=\mathcal{O}(\Delta^{1-\epsilon}).
\end{split}
\end{equation*}

}
\end{rem}

\subsection{\normalsize Strong convergence rate over a finite time interval}

\begin{enumerate}
\item[{\bf (A8)}] There exists a pair of positive constants $l$ and $L$ such that for $x,y,\bar{x},\bar{y}\in\mathbb{R}^n$
\begin{equation*}
\langle x-D(y)-\bar{x}+D(\bar{y}), b(x,y)-b(\bar{x},\bar{y})\rangle\vee\|\sigma(x,y)-\sigma(\bar{x},\bar{y})\|^2\le L(|x-\bar{x}|^2+|y-\bar{y}|^2),
\end{equation*}
and
\begin{equation*}
|b(x,y)-b(\bar{x},\bar{y})|\le L(1+|x|^l+|\bar{x}|^l+|y|^l+|\bar{y}|^l)(|x-\bar{x}|+|y-\bar{y}|).
\end{equation*}
\end{enumerate}

\begin{rem}\label{flow}
{\rm Let (A8) hold. Then there exists a positive constant $\bar{L}$ such that
\begin{equation*}
\langle x-D(y), b(x,y)\rangle\vee\|\sigma(x,y)\|^2\le \bar{L}(1+|x|^2+|y|^2),
\end{equation*}
and
\begin{equation*}
\|\sigma_\Delta(x,y)\|^2\le \bar{L}(1+|x|^2+|y|^2),
\end{equation*}
where $\bar{L}=(L+1)\vee\frac{1}{2}|b(0,0)|^2\vee2\|\sigma(0,0)\|^2$. This implies that (A3) is satisfied. In fact, we can show that (A2) and (A5)-(A6) are also satisfied under assumption (A8).
}
\end{rem}

\begin{lem}\label{bexactbound}
{\rm Let (A1) and (A8) hold, then there exists a positive constant $C$ such that for any $p\ge 2$
\begin{equation*}
\mathbb{E}\left(\sup\limits_{0\le t\le T}|X(t)|^{p}\right)\le C.
\end{equation*}
}
\end{lem}
\begin{proof}
By the It\^{o} formula,
\begin{equation*}
\begin{split}
&|X(t)-D(X(t-\tau))|^{p}\le|\xi(0)-D(\xi(-\tau))|^{p}\\
&+p\int_0^t|X(s)-D(X(s-\tau))|^{p-2}\langle X(s)-D(X(s-\tau)),\sigma(X(s),X(s-\tau)){\mbox d}W(s)\rangle\\
&+p\int_0^t|X(s)-D(X(s-\tau))|^{p-2}\bigg[\langle X(s)-D(X(s-\tau)),b(X(s),X(s-\tau))\rangle\\
&+\frac{p-1}{2}\|\sigma(X(s),X(s-\tau))\|^2\bigg]{\mbox d}s\\
=:&|\xi(0)-D(\xi(-\tau))|^{p}+J_1(t)+J_2(t).
\end{split}
\end{equation*}
By (A8), together with the BDG inequality, we derive
\begin{equation*}
\begin{split}
&\mathbb{E}\left(\sup\limits_{0\le u\le t}|J_1(u)|\right)\le C\mathbb{E}\left(\int_0^t|X(s)-D(X(s-\tau))|^{2p-2}\|\sigma(X(s),X(s-\tau))\|^2{\mbox d}s\right)^{\frac{1}{2}}\\
\le& C\mathbb{E}\left[\left(\sup\limits_{0\le u\le t}|X(u)-D(X(u-\tau))|^{p}\right)\int_0^t|X(s)-D(X(s-\tau))|^{p-2}\|\sigma(X(s),X(s-\tau))\|^2{\mbox d}s\right]^{\frac{1}{2}}\\
\le& \frac{1}{2}\mathbb{E}\left(\sup\limits_{0\le u\le t}|X(u)-D(X(u-\tau))|^{p}\right)+C\mathbb{E}\int_0^t|X(s)-D(X(s-\tau))|^{p-2}\\
&(1+|X(s)|^{2}+|X(s-\tau)|^{2}){\mbox d}s\\
\le& \frac{1}{2}\mathbb{E}\left(\sup\limits_{0\le u\le t}|X(u)-D(X(u-\tau))|^{p}\right)+C+C\int_0^t\mathbb{E}\left(\sup\limits_{0\le u\le s}|X(u)|^{p}\right){\mbox d}s.
\end{split}
\end{equation*}
By (A8) again, we see that
\begin{equation*}
\begin{split}
\mathbb{E}\left(\sup\limits_{0\le u\le t}|J_2(u)|\right)\le& C\mathbb{E}\int_0^t|X(s)-D(X(s-\tau))|^{p-2}(1+|X(s)|^2+|X(s-\tau)|^2){\mbox d}s\\
\le& C+C\int_0^t\mathbb{E}\left(\sup\limits_{0\le u\le s}|X(u)|^{p}\right){\mbox d}s.
\end{split}
\end{equation*}
The Gronwall inequality then implies that
\begin{equation*}
\begin{split}
&\mathbb{E}\left(\sup\limits_{0\le t\le T}|X(t)-D(X(t-\tau))|^{p}\right)\le C.
\end{split}
\end{equation*}
The desired result can be obtained by using (A1).
\end{proof}

With assumption (A8) and Remark \ref{flow}, (A4) can be replaced by the following (A4').
\begin{enumerate}
\item[{\bf (A4')}] There exists a constant $\bar{L}_1>0$ such that
\begin{equation*}
\langle x-D(y), b_\Delta(x,y)\rangle\le \bar{L}_1(1+|x|^2+|y|^2)
\end{equation*}
for any $x,y\in\mathbb{R}^n$.
\end{enumerate}

\begin{lem}\label{water}
{\rm Let (A1), (A4') and (A8) hold, then there exists a positive constant $C$ such that for any $p\ge 2$
\begin{equation}\label{cs}
\mathbb{E}\left(\sup\limits_{0\le t\le T}|Y(t)-\bar{Y}(t)|^{p}\right)\le C\Delta^{p/2}(g(\Delta))^{p},
\end{equation}
and
\begin{equation}\label{sw}
\mathbb{E}\left(\sup\limits_{0\le t\le T}|Y(t)|^{p}\right)\le C.
\end{equation}
}
\end{lem}
\begin{proof}
Similar to the process of Lemma \ref{exa-num}, \eqref{cs} can be verified. Moreover, application of the It\^{o} formula yields
\begin{equation*}
\begin{split}
&\mathbb{E}\left(\sup\limits_{0\le u\le t}|Y(u)-D(\bar{Y}(u-\tau))|^{p}\right)\le |\xi(0)-D(\xi(-\tau))|^p+p\int_0^t|Y(s)-D(\bar{Y}(s-\tau))|^{p-2}\\
&\cdot\left[\langle Y(s)-D(\bar{Y}(s-\tau)),b_\Delta(\bar{Y}(s),\bar{Y}(s-\tau))\rangle+\frac{p-1}{2}\|\sigma_\Delta(\bar{Y}(s),\bar{Y}(s-\tau))\|^2\right]{\mbox d}s\\
&+C\mathbb{E}\bigg(\sup\limits_{0\le u\le t}\bigg|\int_0^u|Y(s)-D(\bar{Y}(s-\tau))|^{p-2}\langle Y(s)-D(\bar{Y}(s-\tau)),\sigma_\Delta(Y(s),Y(s-\tau)){\mbox d}W(s)\rangle\bigg|\bigg)\\
\le&\frac{1}{2}\mathbb{E}\left(\sup\limits_{0\le u\le t}|Y(u)-D(\bar{Y}(u-\tau))|^{p}\right)+C\int_0^t\mathbb{E}\left(\sup\limits_{0\le u\le s}|Y(u)|^{p}\right){\mbox d}s+C\Delta^{p/4}(g(\Delta))^p\\
&+C+C\mathbb{E}\int_0^t|Y(s)-D(\bar{Y}(s-\tau))|^{p-2}\|\sigma_\Delta(Y(s),Y(s-\tau))\|^2{\mbox d}s\\
\le&\frac{1}{2}\mathbb{E}\left(\sup\limits_{0\le u\le t}|Y(u)-D(\bar{Y}(u-\tau))|^{p}\right)+C\int_0^t\mathbb{E}\left(\sup\limits_{0\le u\le s}|Y(u)|^{p}\right){\mbox d}s+C\Delta^{p/4}(g(\Delta))^p\\
&+C+C\mathbb{E}\int_0^t|Y(s)-D(\bar{Y}(s-\tau))|^{p-2}(1+|Y(s)|^{2}+|Y(s-\tau)|^{2}){\mbox d}s\\
\le&\frac{1}{2}\mathbb{E}\left(\sup\limits_{0\le u\le t}|Y(u)-D(\bar{Y}(u-\tau))|^{p}\right)+C\int_0^t\mathbb{E}\left(\sup\limits_{0\le u\le s}|Y(u)|^{p}\right){\mbox d}s+C\Delta^{p/4}(g(\Delta))^p+C,
\end{split}
\end{equation*}
where we have used (A4') and the fact that $\|\sigma_\Delta(x,y)\|^2\le \bar{L}(1+|x|^2+|y|^2)$. Thus, we have
\begin{equation*}
\begin{split}
&\mathbb{E}\left(\sup\limits_{0\le t\le T}|Y(t)-D(\bar{Y}(t-\tau))|^{p}\right)\le C.
\end{split}
\end{equation*}
Finally, \eqref{sw} can be obtained by (A1).

\end{proof}

\begin{lem}\label{interval}
{\rm Let (A1), (A4') and (A8) hold, $R$ is a real number and let $\Delta$ be sufficiently small such that $\|\xi\|_{\infty}<R\le f^{-1}(g(\Delta))$, then there exists a positive constant $C$ independent of $\Delta$ such that for any $q\ge 2$
\begin{equation*}
\mathbb{E}\left(\sup\limits_{0\le t\le T\wedge\nu_R}|X(t)-Y(t)|^q\right)\le C\Delta^{q/2}(g(\Delta))^{q}.
\end{equation*}
}
\end{lem}
\begin{proof}
Use the same notations as before. By the It\^{o} formula,
\begin{equation*}
\begin{split}
&|e(t)|^q\le q\int_0^{t}|e(s)|^{q-2}\bigg[\langle e(s),b(X(s),X(s-\tau))-b_\Delta(\bar{Y}(s),\bar{Y}(s-\tau))\rangle\\
&+\frac{q-1}{2}\|\sigma(X(s),X(s-\tau))-\sigma_\Delta(\bar{Y}(s),\bar{Y}(s-\tau))\|^2\bigg]{\mbox d}s\\
&+q\int_0^{t}|e(s)|^{q-2}\langle e(s),\sigma(X(s),X(s-\tau))-\sigma_\Delta(\bar{Y}(s),\bar{Y}(s-\tau)){\mbox d}W(s)\rangle\\
\le&q\int_0^{t}|e(s)|^{q-2}\bigg[\langle X(s)-D(X(s-\tau))-\bar{Y}(s)+D(\bar{Y}(s-\tau)),b(X(s),X(s-\tau))\\
&-b_\Delta(\bar{Y}(s),\bar{Y}(s-\tau))\rangle+\frac{q-1}{2}\|\sigma(X(s),X(s-\tau))-\sigma_\Delta(\bar{Y}(s),\bar{Y}(s-\tau))\|^2\bigg]{\mbox d}s\\
&+q\int_0^{t}|e(s)|^{q-2}\langle \bar{Y}(s)-Y(s),b(X(s),X(s-\tau))-b_\Delta(\bar{Y}(s),\bar{Y}(s-\tau))\rangle{\mbox d}s\\
&+q\int_0^{t}|e(s)|^{q-2}\langle e(s),\sigma(X(s),X(s-\tau))-\sigma_\Delta(\bar{Y}(s),\bar{Y}(s-\tau)){\mbox d}W(s)\rangle\\
=:&I_1(t)+I_2(t)+I_3(t).
\end{split}
\end{equation*}
Since for $0\le s\le t\wedge\nu_R$, we have $b_\Delta(\bar{Y}(s),\bar{Y}(s-\tau))=b(\bar{Y}(s),\bar{Y}(s-\tau))$ and $\sigma_\Delta(\bar{Y}(s),\bar{Y}(s-\tau))=\sigma(\bar{Y}(s),\bar{Y}(s-\tau))$. Thus, we derive from (A8) and Lemma \ref{water} that
\begin{equation*}
\begin{split}
&\mathbb{E}\left(\sup\limits_{0\le u\le t\wedge\nu_R}|I_1(u)|\right)\le qL_2\mathbb{E}\int_0^{t\wedge\nu_R}|e(s)|^{q-2}(|X(s)-\bar{Y}(s)|^2+|X(s-\tau)-\bar{Y}(s-\tau)|^2){\mbox d}s\\
\le&C\mathbb{E}\int_0^{t\wedge\nu_R}|e(s)|^{q}{\mbox d}s+C\mathbb{E}\int_0^{t\wedge\nu_R}(|X(s)-\bar{Y}(s)|^q+|X(s-\tau)-\bar{Y}(s-\tau)|^q){\mbox d}s \\
\le&C\int_0^{t}\mathbb{E}\left(\sup\limits_{0\le u\le s}|X(u\wedge\nu_R)-Y(u\wedge\nu_R)|^q\right){\mbox d}s+C\Delta^{q/2}(g(\Delta))^q.
\end{split}
\end{equation*}
Moreover, by (A8) and Lemmas \ref{bexactbound}-\ref{water}, we derive
\begin{equation*}
\begin{split}
&\mathbb{E}\left(\sup\limits_{0\le u\le t\wedge\nu_R}|I_2(u)|\right)\le C\mathbb{E}\int_0^{t\wedge\nu_R}|e(s)|^{q}{\mbox d}s+C\mathbb{E}\int_0^{t\wedge\nu_R}|\bar{Y}(s)-Y(s)|^{q}{\mbox d}s\\
&+C\mathbb{E}\int_0^{t\wedge\nu_R}|b(X(s),X(s-\tau))-b(\bar{Y}(s),\bar{Y}(s-\tau))|^{q}{\mbox d}s\\
\le&C\mathbb{E}\int_0^{t\wedge\nu_R}|e(s)|^q{\mbox d}s+C\mathbb{E}\int_0^{t\wedge\nu_R}|\bar{Y}(s)-Y(s)|^{q}{\mbox d}s+C\mathbb{E}\int_0^{t\wedge\nu_R}(|X(s)-\bar{Y}(s)|^{q}\\
&+|X(s-\tau)-\bar{Y}(s-\tau)|^{q})(1+|X(s)|^l+|X(s-\tau)|^l+|\bar{Y}(s)|^l+|\bar{Y}(s-\tau)|^l)^{q}{\mbox d}s\\
\le&C\int_0^{t}\mathbb{E}\left(\sup\limits_{0\le u\le s}|X(u\wedge\nu_R)-Y(u\wedge\nu_R)|^q\right){\mbox d}s+C\Delta^{q/2}(g(\Delta))^{q}.
\end{split}
\end{equation*}
Moreover, by the BDG inequality, we can show that
\begin{equation*}
\begin{split}
&\mathbb{E}\left(\sup\limits_{0\le u\le t\wedge\nu_R}|I_3(u)|\right)\le \frac{1}{2}\mathbb{E}\left(\sup\limits_{0\le u\le t\wedge\nu_R}|e(u)|^{q}\right)\\
&+C\int_0^{t}\mathbb{E}\left(\sup\limits_{0\le u\le s}|X(u\wedge\nu_R)-Y(u\wedge\nu_R)|^q\right){\mbox d}s+C\Delta^{q/2}(g(\Delta))^q.
\end{split}
\end{equation*}
Taking $I_1(t)$-$I_3(t)$ into consideration, we get
\begin{equation*}
\begin{split}
\mathbb{E}\left(\sup\limits_{0\le u\le t\wedge\nu_R}|e(u)|^{q}\right)\le C\int_0^{t}\mathbb{E}\left(\sup\limits_{0\le u\le s}|X(u\wedge\nu_R)-Y(u\wedge\nu_R)|^q\right){\mbox d}s+C\Delta^{q/2}(g(\Delta))^{q}.
\end{split}
\end{equation*}
Consequently,
\begin{equation*}
\begin{split}
&\mathbb{E}\left(\sup\limits_{0\le u\le t\wedge\nu_R}|X(u)-Y(u)|^q\right)\le C+C\mathbb{E}\left(\sup\limits_{0\le u\le t\wedge\nu_R}|e(u)|^{q}\right)\\
\le &C\int_0^{t}\mathbb{E}\left(\sup\limits_{0\le u\le s}|X(u\wedge\nu_R)-Y(u\wedge\nu_R)|^q\right){\mbox d}s+C\Delta^{q/2}(g(\Delta))^{q}.
\end{split}
\end{equation*}
Application of the Gronwall inequality leads to the desired result.

\end{proof}

\begin{thm}
{\rm Let (A1), (A4') and (A8) hold, $\Delta$ is sufficiently small such that
\begin{equation*}
f\left([\Delta^{q/2}(g(\Delta))^{q}]^{\frac{-1}{p-q}}\right)\le g(\Delta),
\end{equation*}
then there exists a positive constant $C$ independent of $\Delta$ such that for $q\ge 2$
\begin{equation*}
\mathbb{E}\left(\sup\limits_{0\le t\le T}|X(t)-Y(t)|^q\right)\le C\Delta^{q/2}(g(\Delta))^{q}.
\end{equation*}
}
\end{thm}
\begin{proof}
By the Young inequality, we derive that for any $\eta>0$,
\begin{equation*}
\begin{split}
&\mathbb{E}\left(\sup\limits_{0\le t\le T}|X(t)-Y(t)|^q\right)=\mathbb{E}\left({\bf I}_{\{\tau_R>T,\rho_R>T\}}\sup\limits_{0\le t\le T}|X(t)-Y(t)|^q\right)\\
&+\frac{q\eta}{p}\mathbb{E}\left(\sup\limits_{0\le t\le T}|X(t)-Y(t)|^{p}\right)+\frac{p-q}{p\eta^{\frac{q}{p-q}}}\mathbb{P}(\tau_R\le T~{\rm or}~\rho_R\le T).
\end{split}
\end{equation*}
By Lemmas \ref{bexactbound}-\ref{water}, we derive
\begin{equation*}
\begin{split}
\mathbb{E}\left(\sup\limits_{0\le t\le T}|X(t)-Y(t)|^q\right)\le&\mathbb{E}\left(\sup\limits_{0\le t\le T\wedge\nu_R}|X(t)-Y(t)|^q\right)+\frac{q\eta C}{p}+\frac{2(p-q)C}{p\eta^{\frac{q}{p-q}}R^p}.
\end{split}
\end{equation*}
Choosing $\eta=\Delta^{q/2}(g(\Delta))^{q}$ and $R=[\Delta^{q/2}(g(\Delta))^{q}]^{\frac{-1}{p-q}}$, noticing that
\begin{equation*}
R=[\Delta^{q/2}(g(\Delta))^{q}]^{\frac{-1}{p-q}}\le f^{-1}(g(\Delta)),
\end{equation*}
applying Lemma \ref{interval}, the desired result follows.
\end{proof}

\section{\large Convergence rates for NSDDEs driven by pure jumps}
In this section, we consider the truncated EM-scheme for NSDDEs driven by pure jumps. Let $(\mathbb{Y},\mathscr{B}(\mathbb{Y}),\lambda(\cdot))$ be a
measurable space, $D_p$ a countable subset of $\mathbb{R}_+$ and $p:D_p\to
\mathbb{Y}$ an adapted process taking values in $\mathbb{Y}$. Then the Poisson random measure
$N(\cdot,\cdot):\mathscr{B}(\mathbb{R}_+\times\mathbb{Y})\times \Omega\rightarrow \mathbb{N}\cup\{0\}$,
defined on the probability space $(\Omega,\mathscr{F},\mathbb{P})$,  can be represented
by
\begin{equation*}
N((0,t]\times\Gamma)=\sum_{s\in D_p,s\le t}{\bf 1}_\Gamma(p(s)),\ \ \
\Gamma\in\mathscr{B}(\mathbb{Y}).
\end{equation*}
In this case, we say that $p$ is a Poisson point process and $N$ is
its associated Poisson random measure. Let $\lambda(\cdot):=\mathbb{E}
N((0,1]\times\cdot)$. Then, the compensated Poisson random measure
\begin{equation*}
\tilde{N}(\mbox{d}u,\mbox{d}t):=N(\mbox{d}u, \mbox{d}t)-\lambda(\mbox{d}u)\mbox{d}t
\end{equation*}
is a martingale. Let $\mathscr{D}:=D([-\tau, 0];\R^n)$ denote the
space of all c\`{a}dl\`{a}g paths $f: [-\tau,0]\rightarrow\mathbb{R}^n$ with the
uniform norm $\|f\|_\infty:=\sup_{-\tau\le\theta\le0}|f(\theta)|$.

In this section, we consider jump-diffusion NSDDE in the form
\begin{equation}\label{jump}
\mbox{d}[X(t)-D(X(t-\tau))]=b(X(t),X(t-\tau))\mbox{d}t+\int_\mathbb{Y} h(X(t),X(t-\tau),u)\tilde{N}(\mbox{d}u,\mbox{d}t),\ \  t\ge0
\end{equation}
with the initial datum
\begin{equation*}
X_0=\xi=\{\xi(\theta):-\tau\le\theta\le 0\}\in \mathcal{L}^p_{\mathscr{F}_0}([-\tau, 0]; \mathbb{R}^n), p\ge 2.
\end{equation*}
Here, $D:\mathbb{R}^n\rightarrow\mathbb{R}^n$, and $b:\mathbb{R}^n\times\mathbb{R}^n\rightarrow\mathbb{R}^n$, $h:\mathbb{R}^n\times\mathbb{R}^n\times\mathbb{Y}\rightarrow\R^n$ are continuous functions. Fix $T>\tau>0$, assume that $T$ and $\tau$ are rational numbers, and the step size $\Delta\in (0,1)$ be fraction of $T$ and $\tau$, so that there exist two positive integers $M, m$ such that $\Delta=T/M=\tau/m$. Assume that $\int_\mathbb{Y}|u|^p\lambda(\mbox{d}u)<\infty$ for $p\ge 1$. In the following paper, we assume that
\begin{enumerate}
\item[{\bf (B1)}] Assume $|h(0,0,u)|\le|u|^p$ and there exists a positive constant $K_1$ such that for all $x,y,\bar{x},\bar{y}\in\R^n$,
\begin{equation*}
\langle x-D(y)-\bar{x}+D(\bar{y}), b(x, y)-b(\bar{x}, \bar{y})\rangle\le K_1(|x-\bar{x}|^2+|y-\bar{y}|^2),
\end{equation*}
\begin{equation*}
\int_\mathbb{Y}|h(x, y,u)-h(\bar{x}, \bar{y},u)|^p\lambda(\mbox{d}u)\le K_1(|x-\bar{x}|^p+|y-\bar{y}|^p), p\ge 2,
\end{equation*}
and
\begin{equation*}
|b(x, y)-b(\bar{x}, \bar{y})|\le K_1(1+|x|^l+|\bar{x}|^l+|y|^l+|\bar{y}|^l)(|x-\bar{x}|+|y-\bar{y}|).
\end{equation*}
\end{enumerate}

\begin{rem}
{\rm With (B1), we see that there exists a positive constant $\bar{K}_1$ such that
\begin{equation*}
\< x-D(y), b(x, y)\>\le \bar{K}_1(1+|x|^2+|y|^2),
\end{equation*}
and
\begin{equation*}
\int_\mathbb{Y}|h(x, y,u)|^p\lambda(\mbox{d}u)\le \bar{K}_1(1+|x|^p+|y|^p), p\ge 2.
\end{equation*}
}
\end{rem}

\begin{lem}
{\rm With (A1) and (B1), there exists a unique solution to \eqref{jump}, and the solution satisfies that
\begin{equation*}
\mathbb{E}\left(\sup\limits_{0\le t \le T}|X(t)|^p\right)\le C,
\end{equation*}
where $C$ is a positive constant.
}
\end{lem}
\begin{proof}
The detained proof can be find in \cite{jy16}.
\end{proof}

\begin{lem}\label{usefullemma}
{\rm Let $\phi:\mathbb{R}_+\times\mathbb{Y}\rightarrow\mathbb{R}^n$ and assume that
\begin{equation*}
\mathbb{E}\int_0^T\int_\mathbb{Y}|\phi(s,u)|^p\lambda(\mbox{d}u)\mbox{d}s<\infty, T\ge 0, p\ge 2.
\end{equation*}
Then, there exists a positive constant $C(p)$ such that
\begin{equation*}
\begin{split}
&\mathbb{E}\left(\sup\limits_{0\le t\le T}\left|\int_0^t\int_\mathbb{Y}\phi(s-,u)\tilde{N}(\mbox{d}u,\mbox{d}s)\right|^p\right)\\
\le &C(p)\left[\mathbb{E}\left(\int_0^T\int_\mathbb{Y}|\phi(s,u)|^2\lambda(\mbox{d}u)\mbox{d}s\right)^{\frac{p}{2}}
+\mathbb{E}\int_0^T\int_\mathbb{Y}|\phi(s,u)|^p\lambda(\mbox{d}u)\mbox{d}s\right].
\end{split}
\end{equation*}
}
\end{lem}
\begin{proof}
See \cite{mpr10} for more details.
\end{proof}

Now we are going to define the truncated EM scheme for \eqref{jump}. Similar to the Brownian motion case, define a strictly increasing continuous function $\bar{f}:\mathbb{R}_+\rightarrow\mathbb{R}_+$ such that $\bar{f}(r)\rightarrow\infty$ as $r\rightarrow\infty$ and for any $r\ge 0$ and $u\in\mathbb{Y}$,
\begin{equation*}
\sup\limits_{|x|\vee|y|\le r}(|b(x,y)|\vee|h(x,y,u)|)\le \bar{f}(r).
\end{equation*}
Define a strictly decreasing function $\bar{g}:(0,1]\rightarrow(0,\infty)$ such that for $p\ge 2$
\begin{equation}\label{jgdelta}
\bar{g}(\Delta)\ge1,~~\lim\limits_{\Delta\rightarrow 0}\bar{g}(\Delta)=\infty,~~\mbox{and} ~~\Delta^{1/4}[\bar{g}(\Delta)]^p\le 1.
\end{equation}
For a given stepsize $\Delta\in(0,1]$, we define the following truncated functions
\begin{equation*}
b_\Delta(x,y)=b\left((|x|\wedge \bar{f}^{-1}(\bar{g}(\Delta)))\frac{x}{|x|},(|y|\wedge \bar{f}^{-1}(\bar{g}(\Delta)))\frac{y}{|y|}\right),
\end{equation*}
and
\begin{equation*}
h_\Delta(x,y,u)=h\left((|x|\wedge \bar{f}^{-1}(\bar{g}(\Delta)))\frac{x}{|x|},(|y|\wedge \bar{f}^{-1}(\bar{g}(\Delta)))\frac{y}{|y|},u\right)
\end{equation*}
for any $x,y\in\mathbb{R}^n$, where we set $x/|x|=0$ for $x=0$. Obviously, for any $x,y\in\mathbb{R}^n$
\begin{equation}\label{jbounddelta}
|b_\Delta(x,y)|\vee|h_\Delta(x,y,u)|\le \bar{g}(\Delta).
\end{equation}
Given any time $T>0$, assume that there exist two positive integers such that $\Delta=\frac{\tau}{m}=\frac{T}{M}$. For $k=-m, \cdots, 0$, set $y_{t_k}=\xi(k\Delta)$; For $k=0, 1, \cdots,M-1$, we form
\begin{equation}\label{jdiscrete}
\begin{split}
y_{t_{k+1}}-D(y_{t_{k+1-m}})=&y_{t_k}-D(y_{t_{k-m}})+b_\Delta(y_{t_{k}}, y_{t_{k-m}})\Delta+h_\Delta(y_{t_{k}}, y_{t_{k-m}},u)\Delta \tilde{N}_k,
\end{split}
\end{equation}
where $t_k=k\Delta$, $\Delta\tilde{N}_k=\tilde{N}(\mathbb{Y},t_{k+1})-\tilde{N}(\mathbb{Y},t_{k})$. Rewrite \eqref{jdiscrete} to a continuous form
\begin{equation}\label{jycontinuous}
\begin{split}
Y(t)-D(\bar{Y}(t-\tau))=&\xi(0)-D(\xi(-\tau))+\int_0^tb_\Delta(\bar{Y}(s), \bar{Y}(s-\tau))\mbox{d}s\\
&+\int_0^t\int_\mathbb{Y} h_\Delta(\bar{Y}(s), \bar{Y}(s-\tau),u)\tilde{N}(\mbox{d}u,\mbox{d}s),
\end{split}
\end{equation}
where $\bar{Y}(t)$ is defined by
\begin{equation*}
\bar{Y}(t):=Y_{t_k} \quad \mbox{for} \quad t\in[t_k, t_{k+1}),
\end{equation*}
thus $\bar{Y}(t-\tau)=Y_{t_{k-m}}$. We further assume
\begin{enumerate}
\item[{\bf (B2)}] There exists a positive constant $K_2$ such that
\begin{equation*}
\< x-D(y), b_\Delta(x, y)\>\le K_2(1+|x|^2+|y|^2).
\end{equation*}
\end{enumerate}

\begin{lem}\label{hdelp}
{\rm Assumption (B1) implies that
\begin{equation*}
\int_\mathbb{Y}|h_\Delta(x, y,u)|^p\lambda(\mbox{d}u)\le \bar{K}_1(1+|x|^p+|y|^p)
\end{equation*}
for $ p\ge 2$.
}
\end{lem}

\begin{lem}\label{jytytk}
{\rm Let (A1), (B1)-(B2) hold, then there exists a positive constant $C$ independent of $\Delta$ such that for $p\ge 2$
\begin{equation*}
\mathbb{E}\left[\sup\limits_{0\le k\le M-1}\sup\limits_{t_k\le t<t_{k+1}}|Y(t)-Y(t_k)|^p\right]\le C\Delta(\bar{g}(\Delta))^p.
\end{equation*}
}
\end{lem}
\begin{proof}
We see from the definition of numerical scheme \eqref{jycontinuous} that for $t\in [t_k, t_{k+1})$,
\begin{equation*}
Y(t)-Y(t_k)=\int_{t_k}^t b_\Delta(\bar{Y}(s), \bar{Y}(s-\tau))\mbox{d}s+\int_{t_k}^t\int_\mathbb{Y} h_\Delta(\bar{Y}(s), \bar{Y}(s-\tau),u)\tilde{N}(\mbox{d}u,\mbox{d}s).
\end{equation*}
By the elementary inequality $|a+b|^p\le 2^{p-1}(|a|^p+|b|^p), p\ge 1$, we compute
\begin{equation*}
\begin{split}
\mathbb{E}\left[\sup\limits_{t_k\le t<t_{k+1}}|Y(t)-Y(t_k)|^p\right]\le &2^{p-1}\mathbb{E}\left[\sup\limits_{t_k\le t<t_{k+1}}\left|\int_{t_k}^t b_\Delta(\bar{Y}(s), \bar{Y}(s-\tau))\mbox{d}s\right|^p\right]\\
&+2^{p-1}\mathbb{E}\left[\sup\limits_{t_k\le t<t_{k+1}}\left|\int_{t_k}^t\int_\mathbb{Y} h_\Delta(\bar{Y}(s), \bar{Y}(s-\tau),u)\tilde{N}(\mbox{d}u,\mbox{d}s)\right|^p\right].
\end{split}
\end{equation*}
By \eqref{jbounddelta}, the H\"{o}lder inequality and the BDG inequality, we get
\begin{equation*}
\begin{split}
\mathbb{E}\left[\sup\limits_{t_k\le t<t_{k+1}}|Y(t)-Y(t_k)|^p\right]
\le& 2^{p-1}\Delta^{p-1}\mathbb{E}\left[\int_{t_k}^{t_{k+1}}\left|b_\Delta(\bar{Y}(s), \bar{Y}(s-\tau))\right|^p\mbox{d}s\right]\\
&+C\mathbb{E}\left[\int_{t_k}^{t_{k+1}}\int_\mathbb{Y} |h_\Delta(\bar{Y}(s), \bar{Y}(s-\tau),u)|^p\lambda(\mbox{d}u)\mbox{d}s\right]\\
\le& C\Delta (\bar{g}(\Delta))^p.
\end{split}
\end{equation*}
This completes the proof.
\end{proof}

\begin{lem}\label{jpmoment}
{\rm
Let (A1), (B1)-(B2) hold, then there exists a positive constant $C$ such that for any $p\ge 2$,
\begin{equation*}\label{jytpmoment}
\mathbb{E}\left(\sup\limits_{0\le t \le T}|Y(t)|^p\right)\le C.
\end{equation*}
}
\end{lem}
\begin{proof}
Applying the It\^{o} formula to $|Y(t)-D(\bar{Y}(t-\tau))|^{p}$, we obtain
\begin{equation*}
\begin{split}
&|Y(t)-D(\bar{Y}(t-\tau))|^{p}=|\xi(0)-D(\xi(-\tau))|^p+p\int_0^{t}|Y(s)-D(\bar{Y}(s-\tau))|^{p-2}\\
&~~~~~~~~~~~~~~~~~~~~~~\langle Y(s)-D(\bar{Y}(s-\tau)), b_\Delta(\bar{Y}(s),\bar{Y}(s-\tau))\rangle\mbox{d}s\\
&+\int_0^{t}\int_\mathbb{Y}[|Y(s)-D(\bar{Y}(s-\tau))+h_\Delta(\bar{Y}(s), \bar{Y}(s-\tau),u)|^p-|Y(s)-D(\bar{Y}(s-\tau))|^p\\
&-p|Y(s)-D(\bar{Y}(s-\tau))|^{p-2}\langle Y(s)-D(\bar{Y}(s-\tau)),h_\Delta(\bar{Y}(s), \bar{Y}(s-\tau),u)\rangle]\lambda(\mbox{d}u)\mbox{d}s\\
&+\int_0^{t}\int_\mathbb{Y}[|Y(s)-D(\bar{Y}(s-\tau))+h_\Delta(\bar{Y}(s), \bar{Y}(s-\tau),u)|^p-|Y(s)-D(\bar{Y}(s-\tau))|^p]\tilde{N}(\mbox{d}u,\mbox{d}s)\\
\le&|\xi(0)-D(\xi(-\tau))|^{p}+p\int_0^t|Y(s)-D(\bar{Y}(s-\tau))|^{p-2}\langle Y(s)-\bar{Y}(s), b_\Delta(\bar{Y}(s),\bar{Y}(s-\tau))\rangle\mbox{d}s\\
&+p\int_0^t|Y(s)-D(\bar{Y}(s-\tau))|^{p-2}\langle\bar{Y}(s)-D(\bar{Y}(s-\tau)), b_\Delta(\bar{Y}(s),\bar{Y}(s-\tau))\rangle\mbox{d}s\\
&+C(p)\int_0^{t}\int_\mathbb{Y}|h_\Delta(\bar{Y}(s), \bar{Y}(s-\tau),u)|^p \lambda(\mbox{d}u)\mbox{d}s\\
&+C(p)\int_0^{t}\int_\mathbb{Y}|Y(s)-D(\bar{Y}(s-\tau))|^{p-2}|h_\Delta(\bar{Y}(s), \bar{Y}(s-\tau),u)|^2\lambda(\mbox{d}u)\mbox{d}s\\
&+\int_0^{t}\int_\mathbb{Y}[|Y(s)-D(\bar{Y}(s-\tau))+h_\Delta(\bar{Y}(s), \bar{Y}(s-\tau),u)|^p-|Y(s)-D(\bar{Y}(s-\tau))|^p]\tilde{N}(\mbox{d}u,\mbox{d}s)\\
:=&|\xi(0)-D(\xi(-\tau))|^{p}+\sum\limits_{i=1}^5E_i(t),
\end{split}
\end{equation*}
where the second step follows because of the Taylor expansion $|a+b|^p-|a|^p-p|a|^{p-2}\langle a,b\rangle\le C(p)(|a|^{p-2}|b|^2+|b|^p)$.
Consequently, by Lemma \ref{jytytk}, we get
\begin{equation*}
\begin{split}
\mathbb{E}\left(\sup\limits_{0\le u\le t}|E_1(u)|\right)\le & C\mathbb{E}\int_0^t|Y(s)-D(\bar{Y}(s-\tau))|^{p}\mbox{d}s\\
&+C\mathbb{E}\int_0^t|Y(s)-\bar{Y}(s)|^{\frac{p}{2}}| b_\Delta(\bar{Y}(s),\bar{Y}(s-\tau))|^{\frac{p}{2}}\mbox{d}s\\
\le&C\int_0^t\mathbb{E}\left(\sup\limits_{0\le u\le s}|Y(u)|^p\right)\mbox{d}s+C+C\Delta(\bar{g}(\Delta))^p.
\end{split}
\end{equation*}
With (B2), Lemma \ref{hdelp} and the H\"{o}lder inequality, we have
\begin{equation*}
\begin{split}
\mathbb{E}\left(\sup\limits_{0\le u\le t}|E_2(u)|\right)+\mathbb{E}\left(\sup\limits_{0\le u\le t}|E_3(u)|\right)\le& C\mathbb{E}\int_0^t(|Y(s)|^{p}+|\bar{Y}(s)|^p+|\bar{Y}(s-\tau)|^p)\mbox{d}s\\
\le&C+C\int_0^t\mathbb{E}\left(\sup\limits_{0\le u\le s}|Y(u)|^p\right)\mbox{d}s.
\end{split}
\end{equation*}
By Lemma \ref{hdelp} and the H\"{o}lder inequality again
\begin{equation*}
\mathbb{E}\left(\sup\limits_{0\le u\le t}|E_4(u)|\right)\le C+C\int_0^t\mathbb{E}\left(\sup\limits_{0\le u\le s}|Y(u)|^p\right)\mbox{d}s.
\end{equation*}
Furthermore, with (B2), Lemmas \ref{usefullemma}-\ref{hdelp}, the Young inequality, the H\"{o}lder inequality and the BDG inequality, we compute
\begin{equation*}
\begin{split}
&\mathbb{E}\left(\sup\limits_{0\le u\le t}|E_5(u)|\right)\le\frac{1}{4}\mathbb{E}\left(\sup\limits_{0\le u\le t}|Y(u)-D(\bar{Y}(u-\tau))|^p\right)\\
&+C\mathbb{E}\int_0^t\int_\mathbb{Y}|h_\Delta(\bar{Y}(s), \bar{Y}(s-\tau),u)|^p \lambda(\mbox{d}u)\mbox{d}s\\
\le&\frac{1}{4}\mathbb{E}\left(\sup\limits_{0\le u\le t}|Y(u)-D(\bar{Y}(u-\tau))|^p\right)+C+C\int_0^t\mathbb{E}\left(\sup\limits_{0\le u\le s}|Y(u)|^p\right)\mbox{d}s.
\end{split}
\end{equation*}
Thus, the estimation of $E_1(t)-E_5(t)$ results in
\begin{equation*}
\begin{split}
\mathbb{E}\left(\sup\limits_{0\le u\le t}|Y(u)-D(\bar{Y}(u-\tau))|^p\right)\le&C+C\int_0^t\mathbb{E}\left(\sup\limits_{0\le u\le s}|Y(u)|^p\right)\mbox{d}s.
\end{split}
\end{equation*}
Finally, the desired result can be obtained by (A1) and the Gronwall inequality.
\end{proof}

\begin{thm}
{\rm Let (A1), (B1)-(B2) hold, $\|\xi\|_\infty<R$ be a real number and $\Delta\in(0,1]$ is sufficiently small such that $R\le \bar{f}^{-1}(\bar{g}(\Delta))$, then for any $q\ge 2$
\begin{equation*}
\mathbb{E}\left(\sup\limits_{0\le t\le T\wedge\nu_R}|X(t)-Y(t)|^q\right)\le C\Delta^{\frac{1}{2}}(\bar{g}(\Delta))^q,
\end{equation*}
where $\nu_R$ is as defined in Section 2.
}
\end{thm}
\begin{proof}
Let $e(t)=X(t)-D(X(t-\tau))-Y(t)+D(\bar{Y}(t-\tau))$. Since for $t\in[0,T\wedge\nu_R]$, we have $b_\Delta(\bar{Y}(t),\bar{Y}(t-\tau))=b(\bar{Y}(t),\bar{Y}(t-\tau))$ and $h_\Delta(\bar{Y}(t), \bar{Y}(t-\tau),u)=h(\bar{Y}(t), \bar{Y}(t-\tau),u)$, denote
\begin{equation*}
\mu(t):=b(X(t),X(t-\tau))-b(\bar{Y}(t),\bar{Y}(t-\tau)),
\end{equation*}
and
\begin{equation*}
\upsilon(t):=h(X(t),X(t-\tau),u)-h(\bar{Y}(t), \bar{Y}(t-\tau),u).
\end{equation*}
Applying the It\^{o} formula,
\begin{equation*}
\begin{split}
|e(t)|^q\le& p\int_0^t |e(s)|^{q-2}\langle e(s),\mu(s)\rangle\mbox{d}s+\int_0^{t}\int_\mathbb{Y}[|e(s)+\upsilon(s)|^q-|e(s)|^q\\
&-p|e(s)|^{q-2}\langle e(s),\upsilon(s)\rangle]\lambda(\mbox{d}u)\mbox{d}s+\int_0^{t}\int_\mathbb{Y}[|e(s)+\upsilon(s)|^q-|e(s)|^q]\tilde{N}(\mbox{d}u,\mbox{d}s)\\
\le&p\int_0^t |e(s)|^{q-2}\langle e(s),\mu(s)\rangle\mbox{d}s+C\int_0^t\int_\mathbb{Y}|e(s)|^{q-2}|\upsilon(s)|^2\lambda(\mbox{d}u)\mbox{d}s\\
&+C\int_0^t\int_\mathbb{Y}|\upsilon(s)|^q\lambda(\mbox{d}u)\mbox{d}s+\int_0^{t}\int_\mathbb{Y}[|e(s)+\upsilon(s)|^q-|e(s)|^q]\tilde{N}(\mbox{d}u,\mbox{d}s)\\
:=&\bar{E}_1(t)+\bar{E}_2(t)+\bar{E}_3(t)+\bar{E}_4(t).
\end{split}
\end{equation*}
With (B2), and Lemmas \ref{jytytk}-\ref{jpmoment}, we calculate
\begin{equation*}
\begin{split}
&\mathbb{E}\left(\sup\limits_{0\le u\le t\wedge\tau_R}|\bar{E}_1(u)|\right)\\
\le& C\mathbb{E}\int_0^{t\wedge\tau_R}|e(s)|^q\mbox{d}s+C\mathbb{E}\int_0^{t\wedge\tau_R} |b(X(s),X(s-\tau))-b(Y(s),Y(s-\tau))|^q\mbox{d}s\\
&+C\mathbb{E}\int_0^{t\wedge\tau_R}|b(Y(s),Y(s-\tau))-b(\bar{Y}(s),\bar{Y}(s-\tau))|^q\mbox{d}s\\
\le& C\int_0^{t\wedge\tau_R}[\mathbb{E}(1+|Y(s)|^{l}+|Y(s-\tau)|^{l}+|X(s)|^{l}+|X(s-\tau)|^{l})^{2q}]^{\frac{1}{2}}\\
&\times[\mathbb{E}(|X(s)-Y(s)|+|X(s-\tau)-Y(s-\tau)|)^{2q}]^{\frac{1}{2}}\mbox{d}s\\
&+C\int_0^{t\wedge\tau_R}[\mathbb{E}(1+|\bar{Y}(s)|^{l}+|\bar{Y}(s-\tau)|^{l}+|Y(s)|^{l}+|Y(s-\tau)|^{l})^{2q}]^{\frac{1}{2}}\\
&\times[\mathbb{E}(|\bar{Y}(s)-Y(s)|+|\bar{Y}(s-\tau)-Y(s-\tau)|)^{2q}]^{\frac{1}{2}}\mbox{d}s\\
&+C\int_0^{t}\mathbb{E}\left(\sup\limits_{0\le u\le s}|X(u\wedge\tau_R)-Y(u\wedge\tau_R)|^q\right)\mbox{d}s\\
\le& C\int_0^{t}\mathbb{E}\left(\sup\limits_{0\le u\le s}|X(u\wedge\tau_R)-Y(u\wedge\tau_R)|^q\right)\mbox{d}s+C\Delta^{\frac{1}{2}}(\bar{g}(\Delta))^{q}.
\end{split}
\end{equation*}
Similarly, we obtain
\begin{equation*}
\begin{split}
&\mathbb{E}\left(\sup\limits_{0\le u\le t\wedge\tau_R}|\bar{E}_2(u)|\right)+\mathbb{E}\left(\sup\limits_{0\le u\le t}|\bar{E}_3(u)|\right)\\
\le& C\int_0^{t}\mathbb{E}\left(\sup\limits_{0\le u\le s}|X(u\wedge\tau_R)-Y(u\wedge\tau_R)|^q\right)\mbox{d}s+C\Delta^{\frac{1}{2}}(\bar{g}(\Delta))^{q}.
\end{split}
\end{equation*}
Furthermore, by Lemmas \ref{jytytk}-\ref{jpmoment}, the BDG inequality and the H\"{o}lder inequality, we compute
\begin{equation*}
\begin{split}
\mathbb{E}\left(\sup\limits_{0\le u\le t\wedge\tau_R}|\bar{E}_4(u)|\right)\le&\frac{1}{4}\mathbb{E}\left(\sup\limits_{0\le u\le t\wedge\tau_R}|X(u)-Y(u)|^q\right)+C\mathbb{E}\int_0^{t\wedge\tau_R}\int_\mathbb{Y}|\upsilon(s)|^q \lambda(\mbox{d}u)\mbox{d}s\\
\le&\frac{1}{4}\mathbb{E}\left(\sup\limits_{0\le u\le t\wedge\tau_R}|X(u)-Y(u)|^q\right)+C\Delta^{\frac{1}{2}}(\bar{g}(\Delta))^{q}\\
&+C\int_0^{t}\mathbb{E}\left(\sup\limits_{0\le u\le s}|X(u\wedge\tau_R)-Y(u\wedge\tau_R)|^q\right)\mbox{d}s.
\end{split}
\end{equation*}
Consequently, by sorting the estimation of $\bar{E}_1(t)-\bar{E}_4(t)$, we arrive at
\begin{equation*}
\begin{split}
\mathbb{E}\left(\sup\limits_{0\le u\le t\wedge\tau_R}|e(u)|^q\right)\le C\int_0^{t}\mathbb{E}\left(\sup\limits_{0\le u\le s}|X(u\wedge\tau_R)-Y(u\wedge\tau_R)|^q\right)\mbox{d}s+C\Delta^{\frac{1}{2}}(\bar{g}(\Delta))^{q}.
\end{split}
\end{equation*}
Thus, by (A1),
\begin{equation*}
\begin{split}
&\mathbb{E}\left(\sup\limits_{0\le u\le t\wedge\tau_R}|X(u)-Y(u)|^q\right)\le\mathbb{E}\left(\sup\limits_{0\le u\le t\wedge\tau_R}|e(u)|^q\right)\\
\le& C\int_0^{t}\mathbb{E}\left(\sup\limits_{0\le u\le s\wedge\tau_R}|X(u)-Y(u)|^q\right)\mbox{d}s+C\Delta^{\frac{1}{2}}(\bar{g}(\Delta))^{q}.
\end{split}
\end{equation*}
The desired result follows by the Gronwall inequality.
\end{proof}

\begin{thm}
{\rm Let (A1), (B1)-(B2) hold, $\Delta$ is sufficiently small such that
\begin{equation*}
f\left([\Delta^{1/2}(\bar{g}(\Delta))^{q}]^{\frac{-1}{p-q}}\right)\le \bar{g}(\Delta),
\end{equation*}
then there exists a positive constant $C$ independent of $\Delta$ such that for $p\ge 2$
\begin{equation*}
\mathbb{E}\left(\sup\limits_{0\le t\le T}|X(t)-Y(t)|^p\right)\le C\Delta^{1/2}(\bar{g}(\Delta))^{p}.
\end{equation*}
}
\end{thm}

\subsection*{Appendix}
{\bf 1. Proof of \eqref{remmm} in Remark \ref{example1}.}  By the definition of $b_\Delta$ and $\sigma_\Delta$, it is easy to see for $(|x|\vee|y|)\le f^{-1}(g(\Delta))$, we have $b_\Delta(x,y)=b(x,y)$ and $\sigma_\Delta(x,y)=\sigma(x,y)$, thus with assumption (A3),
\begin{equation*}
\langle x, b_\Delta(x,y)\rangle+\frac{p-1}{2}\|\sigma_\Delta(x,y)\|^2=\langle x, b(x,y)\rangle+\frac{p-1}{2}\|\sigma(x,y)\|^2\le L_1(1+|x|^2+|y|^2).
\end{equation*}
For $(|x|\wedge|y|)>f^{-1}(g(\Delta))$, we have
\begin{equation*}
\begin{split}
&\langle x, b_\Delta(x,y)\rangle+\frac{p-1}{2}\|\sigma_\Delta(x,y)\|^2\\
=&\left(\frac{|x|}{f^{-1}(g(\Delta))}-1\right)\left\langle f^{-1}(g(\Delta))\frac{x}{|x|}, b\left(f^{-1}(g(\Delta))\frac{x}{|x|},f^{-1}(g(\Delta))\frac{y}{|y|}\right)\right\rangle\\
&+\left\langle f^{-1}(g(\Delta))\frac{x}{|x|}, b\left(f^{-1}(g(\Delta))\frac{x}{|x|},f^{-1}(g(\Delta))\frac{y}{|y|}\right)\right\rangle\\
&+\frac{p-1}{2}\left\|\sigma\left(f^{-1}(g(\Delta))\frac{x}{|x|},f^{-1}(g(\Delta))\frac{y}{|y|}\right)\right\|^2\\
\le&\left(\frac{|x|}{f^{-1}(g(\Delta))}-1\right)L_1\left(1+2[f^{-1}(g(\Delta))]^2\right)+L_1(1+2[f^{-1}(g(\Delta))]^2)\\
\le& 2L_1(1+|x|^2+|y|^2),
\end{split}
\end{equation*}
where we have used the fact that $f^{-1}(g(\Delta))\ge f^{-1}(g(\Delta^*))\ge f^{-1}(f(2))=2$. Meanwhile, for $|x|>f^{-1}(g(\Delta))$ and $|y|\le f^{-1}(g(\Delta))$,
\begin{equation*}
\begin{split}
&\langle x, b_\Delta(x,y)\rangle+\frac{p-1}{2}\|\sigma_\Delta(x,y)\|^2\\
=&\left(\frac{|x|}{f^{-1}(g(\Delta))}-1\right)\left\langle f^{-1}(g(\Delta))\frac{x}{|x|}, b\left(f^{-1}(g(\Delta))\frac{x}{|x|},y\right)\right\rangle\\
&+\left\langle f^{-1}(g(\Delta))\frac{x}{|x|}, b\left(f^{-1}(g(\Delta))\frac{x}{|x|},y\right)\right\rangle
+\frac{p-1}{2}\left\|\sigma\left(f^{-1}(g(\Delta))\frac{x}{|x|},y\right)\right\|^2\\
\le&\left(\frac{|x|}{f^{-1}(g(\Delta))}-1\right)L_1\left(1+[f^{-1}(g(\Delta))]^2+|y|^2\right)+L_1(1+[f^{-1}(g(\Delta))]^2+|y|^2)\\
\le& L_1|x|\left(\frac{1}{f^{-1}(g(\Delta))}+f^{-1}(g(\Delta))+\frac{|y|^2}{f^{-1}(g(\Delta))}\right)\le 2L_1(1+|x|^2+|y|^2).
\end{split}
\end{equation*}
Similarly, the result can be obtained for $|x|\le f^{-1}(g(\Delta))$ and $|y|>f^{-1}(g(\Delta))$.\\
{\bf 2. Proof in Remark \ref{example2}.} For $D(y)=\frac{1}{2}\sin y$, $b(x,y)=x-x^3+\cos y$, $\sigma(x,y)=|x|^{\frac{3}{2}}$. It is obvious that (A1) is satisfied with $\kappa=\frac{1}{2}$. By computation, for $|x|\vee|y|\vee|\bar{x}|\vee|\bar{y}|\le R$,
\begin{equation*}
\begin{split}
|b(x,y)-b(\bar{x},\bar{y})|=|x-x^3+\cos y-\bar{x}+\bar{x}^3-\cos \bar{y}|\le(3R^2+1)(|x-\bar{x}|+|y-\bar{y}|),
\end{split}
\end{equation*}
and
\begin{equation*}
\begin{split}
\|\sigma(x,y)-\sigma(\bar{x},\bar{y})\|=\left||x|^{\frac{3}{2}}-|\bar{x}|^{\frac{3}{2}}\right|\le 3\sqrt{R}(|x-\bar{x}|+|y-\bar{y}|).
\end{split}
\end{equation*}
Furthermore, for $p=3$, we have
\begin{equation*}({\bf A.1})
\begin{split}
&\langle x-D(y), b(x,y)\rangle+\|\sigma(x,y)\|^2=\left(x-\frac{1}{2}\sin y\right)(x-x^3+\cos y)+|x|^3\le2(1+|x|^2+|y|^2).
\end{split}
\end{equation*}
To show (A4), we have to divide it into several cases. \\
{\bf Case 1:} $(|x|\vee|y|)\le f^{-1}(g(\Delta))$, we have $b_\Delta(x,y)=b(x,y)$ and $\sigma_\Delta(x,y)=\sigma(x,y)$, thus with ({\bf A.1}),
\begin{equation*}
\langle x-D(y), b_\Delta(x,y)\rangle+\|\sigma_\Delta(x,y)\|^2=\langle x-D(y), b(x,y)\rangle+\|\sigma(x,y)\|^2\le 2(1+|x|^2+|y|^2).
\end{equation*}
{\bf Case 2:} $(|x|\wedge|y|)>f^{-1}(g(\Delta))$, we derive
\begin{equation*}
\begin{split}
&\langle x-D(y), b_\Delta(x,y)\rangle+\|\sigma_\Delta(x,y)\|^2=\left|f^{-1}(g(\Delta))\frac{x}{|x|}\right|^3\\
&+\left(x-\frac{1}{2}\sin y\right)\left[f^{-1}(g(\Delta))\frac{x}{|x|}-\left(f^{-1}(g(\Delta))\frac{x}{|x|}\right)^3+\cos\left(f^{-1}(g(\Delta))\frac{y}{|y|}\right)\right]\\
\le&\frac{3}{2}+2|x|^2-\frac{[f^{-1}(g(\Delta))]^3}{|x|^3}x^4+\frac{3[f^{-1}(g(\Delta))]^4}{4|x|^4}x^4+\frac{[f^{-1}(g(\Delta))]^3}{|x|^3}|x|^3\\
\le& 2(1+|x|^2+|y|^2),
\end{split}
\end{equation*}
where we have used the fact that $f^{-1}(g(\Delta))<|x|$.\\
{\bf Case 3:} $|x|>f^{-1}(g(\Delta))$ and $|y|\le f^{-1}(g(\Delta))$, we have
\begin{equation*}
\begin{split}
&\langle x-D(y), b_\Delta(x,y)\rangle+\|\sigma_\Delta(x,y)\|^2=\left|f^{-1}(g(\Delta))\frac{x}{|x|}\right|^3+\left(x-\frac{1}{2}\sin y\right)\\
&~~~~\cdot\left[f^{-1}(g(\Delta))\frac{x}{|x|}-\left(f^{-1}(g(\Delta))\frac{x}{|x|}\right)^3+\cos y\right]\le 2(1+|x|^2+|y|^2).
\end{split}
\end{equation*}
{\bf Case 4: } $|x|\le f^{-1}(g(\Delta))$ and $|y|>f^{-1}(g(\Delta))$. The process is similar to that of Case 3.


\begin{thebibliography}{20}
\bibitem{by13} Bao J.H., Yuan C.G., Convergence rate of EM scheme for SDDEs, {\it P. American Math. Soc.}, {\bf 141} (2013), 3231-3243.

\bibitem{hms02} Higham D.J., Mao X.R., Stuart A.M., Strong convergence of Euler-type methods for nonlinear stochastic differential equations, {\it SIAM J. Numer. Anal.}, {\bf 40} (2002), 1041-1063.


\bibitem{hms03} Higham D.J., Mao X.R., Stuart A.M., Exponential mean-square stability of numerical solutions to stochastic differential equations, {\it LMS J. Comput. Math.}, {\bf 6} (2003), 297-313.


\bibitem{hjk12} Hutzenthaler M., Jentzen A., Kloeden P.E., Strong convergence of an explicit numerical method for SDEs with nonglobally Lipschitz continuous coefficients, {\it Annals of Appl. Prob.}, {\bf 22} (2012), 1611-1641.


\bibitem{jy16} Ji Y.T., Yuan C.G., Tamed EM scheme of neutral stochastic differential delay equations, arXiv: 1603.06747, 2016.

\bibitem{ks14} Kumar C., Sabanis S., Strong convergence of Euler approximations of stochastic differential equations with delay under local Lipschitz condition, {\it Stoch. Anal. Appl.}, {\bf 32}, (2014), 207-228.

\bibitem{m03} Mao X.R., Numerical solutions of stochastic functional differential equations, {\it LMS J. Comput. Math.}, {\bf 6} (2003), 141-161.

\bibitem{m15} Mao X.R., The truncated Euler-Maruyama method for stochastic differential equations, {\it J. Compu. Appl. Math.}, {\bf 290} (2015), 370-383.


\bibitem{m16} Mao X.R., Convergence rates of the truncated Euler-Maruyama method for stochastic differential equations, {\it J. Compu. Appl. Math.}, {\bf 296} (2016), 362-375.


\bibitem{ms13} Mao X.R., Szpruch, L., Strong convergence rates for backward Euler-Maruyama method for non-linear dissipative-type stochastic differential equations with super-linear diffusion coefficients, {\it Stochastics}, {\bf 85} (2013), 144-171.

\bibitem{mpr10} Marinelli C., Pr\'{e}v\^{o}t C., R\"{o}ckner M., Regular dependence on initial data for stochastic evolution equations with multiplicative Poisson noise, {\it J. Funct. Anal.}, {\bf 258} (2010), 616-649.

\bibitem{mm15} Milo\v{s}evi\'{c} M., Convergence and almost sure exponential stability of implicit numerical methods for a class of highly nonlinear neutral stochastic differential equations with constant delay, {\it J. Comp. Appl. Math.}, {\bf 280} (2015), 248-264.

\bibitem{s13} Sabanis S., A note on tamed Euler approximations, {\it Electron. Commun. Probab.}, {\bf 18} (2013), 1-10.

\bibitem{s15} Sabanis S., Euler approximations with varying coefficients: the case of superlinearly growing diffusion coefficients, {\it Annal. Appl. Probab.}, {\bf 26} (2016), 2083-2105.

\bibitem{ty162} Tan L., Yuan C.G., Convergence rates of $\theta$-method for neutral SDDEs under non-globally Lipschitz continuous coefficients, arXiv:
1701.00223v1, 2016.

\bibitem{zwh152} Zong X.F., Wu F.K., Huang C.M., Theta schemes for SDDEs with non-globally Lipschitiz continuous coefficients, {\it J. Comp. Appl. Math.}, {\bf 278} (2015), 258-277.


\bibitem{zwh15} Zong X.F., Wu F.K., Huang C.M., Exponential mean square stability of the theta approximations for neutral stochastic differential delay equations, {\it J. Comp. Appl. Math.}, {\bf 286} (2015), 172-185.

\end{thebibliography}
\end{document}